\documentclass[10pt]{article}

\usepackage{amssymb}
\usepackage{amsmath}
\usepackage{algpseudocode}
\usepackage{graphics,epsfig}
\usepackage{url,here}
\usepackage[backref,colorlinks=true]{hyperref}
\usepackage{multirow}
\usepackage{algorithm}
\usepackage{placeins}
\DeclareMathOperator*{\argmin}{arg\,min}
\allowdisplaybreaks

\begin{document}
\newenvironment {proof}{{\noindent\bf Proof.}}{\hfill $\Box$ \medskip}

\newtheorem{theorem}{Theorem}[section]
\newtheorem{lemma}[theorem]{Lemma}
\newtheorem{condition}[theorem]{Condition}
\newtheorem{proposition}[theorem]{Proposition}
\newtheorem{remark}[theorem]{Remark}
\newtheorem{definition}[theorem]{Definition}
\newtheorem{hypothesis}[theorem]{Hypothesis}
\newtheorem{corollary}[theorem]{Corollary}
\newtheorem{example}[theorem]{Example}
\newtheorem{descript}[theorem]{Description}
\newtheorem{assumption}[theorem]{Assumption}

\newcommand{\ba}{\begin{align}}
\newcommand{\ea}{\end{align}}

\def\P{\mathbb{P}}
\def\R{\mathbb{R}}
\def\E{\mathbb{E}}
\def\N{\mathbb{N}}
\def\Z{\mathbb{Z}}

\renewcommand {\theequation}{\arabic{section}.\arabic{equation}}
\def \non{{\nonumber}}
\def \hat{\widehat}
\def \tilde{\widetilde}
\def \bar{\overline}

\def\ind{{\mathchoice {\rm 1\mskip-4mu l} {\rm 1\mskip-4mu l}
{\rm 1\mskip-4.5mu l} {\rm 1\mskip-5mu l}}}

\title{\Large\ {\bf An efficient and unbiased method for sensitivity analysis of stochastic reaction networks}}

\author{Ankit Gupta and Mustafa Khammash\\
Department of Biosystems Science and Engineering \\ ETH Zurich \\  Mattenstrasse 26 \\ 4058 Basel, Switzerland. 
}
\date{}
\maketitle
\begin{abstract}
We consider the problem of estimating parameter sensitivity for Markovian models of reaction networks. Sensitivity values measure the responsiveness of an output with respect to the model parameters. They help in analyzing the network, understanding its robustness properties and identifying the important reactions for a specific output. Sensitivity values are commonly estimated using methods that perform finite-difference computations along with Monte Carlo simulations of the reaction dynamics. These methods are computationally efficient and easy to implement, but they produce a \emph{biased} estimate which can be unreliable for certain applications. Moreover the size of the bias is generally unknown and hence the accuracy of these methods cannot be easily determined.  There also exist unbiased schemes for sensitivity estimation but these schemes can be computationally infeasible, even for very simple networks. Our goal in this paper is to present a new method for sensitivity estimation, which combines the computational efficiency of finite-difference methods with the accuracy of unbiased schemes.  Our method is easy to implement and it relies on an exact representation of parameter sensitivity that we recently proved in an earlier paper. Through examples we demonstrate that the proposed method can outperform the existing methods, both biased and unbiased, in many situations.
\end{abstract}

 
\medskip

\setcounter{equation}{0}

\section{Introduction}

Reaction networks are frequently encountered in several scientific disciplines, such as, Chemistry \cite{BookCR}, Systems Biology \cite{CellSignaling,MO}, Epidemiology \cite{Hethcote}, Ecology \cite{Bascompte} and Pharmacology \cite{Berger}.
It is well-known that when the population sizes of the reacting species are small, then a deterministic formulation of the reaction dynamics is inadequate in understanding many important properties of the system (see \cite{Elowitz,ArkinAdams}). Hence stochastic models are necessary and the most prominent class of such models is where the reaction dynamics is described by a continuous time Markov chain. Our paper is concerned with sensitivity analysis of these Markovian models of reaction networks.

Generally, reaction networks involve many kinetic parameters that influence their dynamics. With sensitivity analysis one can measure the receptiveness of an outcome with respect to small changes in parameter values. These sensitivity values give insights about network design and its robustness properties \cite{Stelling}. They also help in identifying the important reactions for a given output, estimating the model parameters and in fine-tuning the system's behaviour \cite{Feng}. The existing literature contains many methods for sensitivity analysis of stochastic reaction networks, but all these methods have certain drawbacks. They either introduce a \emph{bias} in the sensitivity estimate \cite{IRN,KSR1,KSR2,DA} or the associated estimator can have large variances \cite{Gir} or the method becomes impractical for large networks \cite{Our}. Due to these reasons, the search for better methods for sensitivity analysis of stochastic reaction networks is still an active research problem.

We now formally describe a stochastic reaction network consisting of $d$ species ${\bf S}_1,\dots,{\bf S}_d$. Under the well-stirred assumption \cite{GP}, the state of the system at any time is given by a vector in $\N^d_0$, where $\N_0$ is the set of non-negative integers. When the state is $x=(x_1,\dots,x_d)$ then the population size corresponding to species ${\bf S}_i$ is $x_i$. The state evolves as the species interact through $K$ reaction channels. If the state is $x$, then the $k$-th reaction fires at rate $\lambda_k(x)$ and it displaces the state by the \emph{stoichiometric vector} $\zeta_k \in \Z^d$. Here $\lambda_k$ is called the \emph{propensity} function for the $k$-th reaction. We can represent the reaction dynamics by a continuous time Markov chain and the distribution of this process evolves according to the chemical master equation \cite{GP} which is quite difficult to solve except in some restrictive cases. Fortunately, the sample paths of this process can be easily simulated using Monte Carlo procedures such as Gilespie's \emph{stochastic simulation algorithm} (SSA) and its variants \cite{GP,NR}.

Suppose that the propensity functions of the reaction network depend on a scalar parameter $\theta$. Hence when the state is $x$, the $k$-th reaction fires at rate $\lambda_k(x,\theta)$. Let $(X_\theta(t))_{t \geq 0}$ denote the Markov process representing the reaction dynamics. For a function $f: \N^d_0 \to \R$ and an observation time $T \geq 0$, our output of interest is $f(X_\theta(T))$. We would like to determine how much the expected value of this output changes with infinitesimal changes in the parameter $\theta$. In other words, our aim is to compute
\begin{align}
\label{def:sensitivity}
S_\theta(f,T) = \frac{ \partial  }{ \partial \theta } \E\left( f( X_\theta(T) ) \right).
\end{align}
Since the mapping $\theta \mapsto \E\left( f( X_\theta(T) ) \right)$ is generally unknown, it is difficult to evaluate $S_\theta(f,T) $ directly. Therefore we need to estimate this quantity using simulations of the process $X_\theta$. For this purpose, many methods use a finite-difference scheme such as 
\begin{align}
\label{fd:form}
S_{\theta,h}(f,T) = \frac{1}{h} \E \left(  f( X_{\theta+h}(T) ) - f( X_\theta(T) ) \right)
\end{align}
for a small $h$. These methods reduce the variance of the associated estimator by coupling the processes $X_\theta$ and $X_{\theta+h}$ in an intelligent way. Three such couplings are : Common Reaction Numbers (CRN) (see \cite{KSR1}), Common Reaction Paths (CRP) (see \cite{KSR1}) and Coupled finite-differences (CFD) (see \cite{DA}). The two best performing finite-difference schemes are CRP and CFD and we shall discuss them in greater detail in Section \ref{subsec:fdschemes}. It is immediate that a finite-difference scheme produces a \emph{biased} estimate of the true sensitivity value $S_{\theta}(f,T) $. This problem of bias is compounded by the fact that in most cases, the size and even the sign of the bias is unknown, thereby casting a doubt over the estimated sensitivity values. The magnitude of the bias must be proportional to $h$, but it can still be significant for a small $h$ (see Section \ref{sec:ex}). One can reduce the bias by decreasing $h$, but as $h$ gets smaller, the variance of the finite-difference estimator gets larger, making it necessary to generate an extremely large number of samples to obtain a statistically accurate estimate. In many cases, for small values of $h$, the computational cost of generating a large sample is so high that finite-difference schemes become inefficient in comparison to the unbiased methods. Unfortunately there is no way to select an optimum value of $h$, which is small enough to ensure that the bias is small and simultaneously large enough to ensure that the estimator variance is low. This is the main source of difficulty in using finite-difference schemes and we shall revisit this point in Section \ref{sec:ex}. 

To obtain highly accurate estimates of $S_\theta(f,T)$ we need unbiased methods for sensitivity analysis. The first such method was given by Plyasunov and Arkin \cite{Gir} and it relies on the Girsanov measure transformation. Hence we refer to it as the \emph{Girsanov method} in this paper. In this method the sensitivity value $S_\theta(f,T)$ is expressed as
\begin{align}
\label{sens:exact}
S_{\theta}(f,T) = \E \left( s_\theta(f,T)\right),
\end{align}
where $s_\theta(f,t)$ is a random variable whose realizations can be easily obtained by simulating paths of the process $X_\theta$ in the time interval $[0,T]$. The Girsanov method is easy to implement, and its unbiasedness guarantees convergence to the right value $S_{\theta}(f,T) $ as the sample size tends to infinity. However in many situations, the estimator associated with this method has a very high variance, making it extremely inefficient (see the examples in \cite{KSR1,KSR2,DA}). One such situation that commonly arises is when the sensitive parameter $\theta$ is a reaction rate constant with a small size (see \cite{Our}). To remedy this problem of high estimator variances, we proposed another unbiased scheme in \cite{Our}, which is based on a different sensitivity formula of the form \eqref{sens:exact}. In this formula, the expression for the random variable $s_\theta(f,T)$ is derived using the random time-change representation of Kurtz (see Chapter 7 in \cite{EK}). Unfortunately this expression is such that realizations of the random variable $s_\theta(f,T)$ cannot be easily computed from the paths of the process $X_{\theta}$ as the expression involves several expectations of functions of the underlying Markov process at various times and various initial states. If all these expectations are estimated \emph{serially} using independent paths then the problem becomes computationally intractable. Hence we devised a scheme in \cite{Our}, called the \emph{Auxiliary Path Algorithm}(APA), to estimate all these expectations in \emph{parallel} using a fixed number of \emph{auxiliary} paths. The implementation of APA is quite difficult and the memory requirements are very high because one needs to store the paths and dynamically maintain a large table to estimate the relevant quantities. These reasons make APA impractical for large networks and also for large observation times $T$.  However, in spite of these difficulties, we showed in \cite{Our} that APA can be far more efficient than the Girsanov method, in examples where sensitivity is computed with respect to a small reaction rate constant.

The above discussion suggests that all the existing methods for sensitivity analysis, both biased and unbiased, have certain drawbacks. Motivated by this issue, we develop a new method for sensitivity estimation in this paper which is unbiased, easy to implement, has low memory requirements and is more versatile than the existing unbiased schemes. We use the main result in \cite{Our}, to construct another random variable $\hat{s}_\theta(f,T)$ such that \eqref{sens:exact} holds. We then provide a simple scheme, called the \emph{Poisson Path Algorithm}(PPA), to obtain realizations of $\hat{s}_\theta(f,T)$ and this gives us a method to obtain unbiased estimates of parameter sensitivities for stochastic reaction networks.  Similar to APA, PPA also requires estimation of certain quantities using auxiliary paths, but the number of these quantities can be controlled to be so low, that they can be estimated serially using independent paths. Consequently PPA does not require any storage of paths or the dynamic maintenance of a large table as in APA. Hence the memory requirements are low, the implementation is easier, and PPA works well for large networks and for large observation times $T$. In Section \ref{sec:ex} we consider many examples to compare the performance of PPA with the Girsanov method, CFD and CRP. We find that PPA is usually far more efficient than the Girsanov method. Perhaps surprisingly, in many situations PPA can also outperform the finite-difference schemes (CRP and CRN) if we impose the restriction that the bias is small. This makes PPA an attractive method for sensitivity analysis because one can efficiently obtain sensitivity estimates and not have to worry about the (unknown) bias caused by finite-difference approximations.

This paper is organized as follows. In Section \ref{sec:pre} we provide the relevant mathematical background and also discuss the existing schemes for sensitivity analysis in greater detail. We also present our main result which shows that for a suitably defined random variable $\hat{s}_\theta(f,T)$, relation \eqref{sens:exact} is satisfied. In Section \ref{sec:ppa} we describe PPA which is a simple method to obtain realizations of $\hat{s}_\theta(f,T)$. In Section \ref{sec:ex} we consider many examples to illustrate the efficiency of PPA and compare its performance with other methods. Finally in Section \ref{sec:conc} we conclude and discuss future research directions.

\section{Preliminaries}\label{sec:pre}

Recall the description of the reaction network from the previous section and suppose that the propensity functions depend on a real-valued parameter $\theta$. We can model the reaction dynamics by a continuous time Markov process whose generator\footnote{The generator of a Markov process is an operator which specifies the rate of change of the distribution of the process. See Chapter 4 in \cite{EK} for more details.} is given by
\begin{align*}
\mathbb{A}_\theta f(x) = \sum_{k=1}^K \lambda_k(x,\theta) \Delta_{\zeta_k} f(x),
\end{align*}
where $f : \N_0^d \to \R$ is any bounded function and $\Delta_{\zeta}f(x) = f(x+\zeta)-f(x)$. Under some mild conditions on the propensity functions (see Condition 2.1 in \cite{Our}), a Markov process $(X_\theta(t))_{t \geq 0}$ with generator $\mathbb{A}_\theta$ and any initial state $x_0$, exists uniquely.
The random time-change representation of Kurtz (see Chapter 7 in \cite{EK}) shows that this process can be expressed as
\begin{align}
\label{main:rtcrep}
X_\theta(t) = x_0 + \sum_{k=1}^K Y_k\left( \int_{0}^{t}  \lambda_k( X_\theta(s) ,\theta  ) ds \right) \zeta_k,
\end{align}
where $\{Y_k : k = 1, . . . ,K\}$ is a family of independent unit rate Poisson processes.

In this paper we are interested in computing the sensitivity value $S_{\theta}(f,T)$ defined by \eqref{def:sensitivity}. Assume that we can express the sensitivity value as \eqref{sens:exact} for some random variable $s_\theta(f,T)$. If we are able to generate $N$ independent samples $s^{(1)}_\theta(f,T) , \dots, s^{(N)}_\theta(f,T) $ from the distribution of $s_\theta(f,T) $, then $S_{\theta}(f,T)$ can be estimated as
\begin{align}
\label{sens_est1}
\hat{S}_\theta(f,T) = \frac{1}{N} \sum_{i = 1}^N s^{(i)}_{\theta}(f,T),
\end{align}
which is a random variable with mean $\mu_N$ and variance $\sigma^2_N$ given by 
\begin{align}
\label{sens_est_var}
\mu_N = \E\left(  \frac{1}{N} \sum_{i = 1}^N s^{(i)}_\theta(f,T)  \right) = \E(s_\theta(f,T) )  \ \textnormal{  and  } \   \sigma^2_N = \mathrm{Var} \left( \frac{1}{N} \sum_{i = 1}^N s^{(i)}_\theta(f,T)  \right) = \frac{1}{N} \mathrm{Var}(s_\theta(f,T) ).
\end{align}
Due to the Central Limit Theorem, for large values of $N$, the distribution of $\hat{S}_\theta(f,T)$ is approximately normal with mean $\mu_N$ and variance $\sigma^2_N$. Hence for any interval $[a,b] \subset \R$ the probability $\P(\hat{S}_\theta(f,T) \in [a,b])$ can be approximated as
\begin{align}
\label{probestimate}
\P\left(  \hat{S}_\theta(f,T) \in [a,b] \right) \approx \Phi\left( \frac{b - \mu_N}{ \sigma_N } \right) -\Phi\left( \frac{a - \mu_N}{ \sigma_N } \right),
\end{align}
where $\Phi(\cdot)$ is the cumulative distribution function for the standard normal distribution and $\sigma_N$ is the standard deviation. Generally $\mu_N$ and $\sigma_N$ are unknown but they can be estimated from the sample as
\begin{align*}
\mu_N \approx \hat{S}_\theta(f,T)  \  \textnormal{  and  }  \  \sigma_N \approx   \sqrt{\frac{1}{N(N-1)} \sum_{i = 1}^N  \left(  s^{(i)}_\theta(f,T)  - \hat{S}_\theta(f,T)  \right)^2}.
\end{align*}

Note that the mean $\mu_N$ is just the \emph{one-point estimate} for the sensitivity value while the standard deviation $\sigma_N$ measures the statistical spread around $\mu_N$. The standard deviation $\sigma_N$ can also be seen as the estimation error. From \eqref{sens_est_var} it is immediate that  $\sigma_N$ is directly proportional to $\mathrm{Var}(s_\theta(f,T) )$ but inversely proportional to $N$. Hence if $\mathrm{Var}(s_\theta(f,T) )$ is high then a larger number of samples is needed to achieve a certain statistical accuracy.

In finite-difference schemes, instead of $S_{\theta}(f,T)$ one estimates a finite-difference $S_{\theta,h}(f,T)$ of the form \eqref{fd:form} which can be written as $S_{\theta,h}(f,T) = \E\left( s_{\theta,h}(f,T)  \right)$ for 
\begin{align}
\label{fd:form2}
s_{\theta,h}(f,T) = \frac{ f( X_{\theta +h} (T)  )  - f( X_{\theta } (T)  )   }{h}.
\end{align}
Here the Markov processes $X_\theta$ and $X_{\theta+h}$ are defined on the same probability space and they have generators $\mathbb{A}_{\theta}$ and $\mathbb{A}_{\theta+h}$ respectively. In this case, the estimator mean $\mu_N$ and variance $\sigma^2_N$ are given by \eqref{sens_est_var} with $s_{\theta}(f,T)$ replaced by $s_{\theta,h}(f,T)$.

\subsection{Finite-difference schemes}\label{subsec:fdschemes}

The main idea behind the various finite-difference schemes is that by coupling the processes $X_\theta$ and $X_{\theta+h}$, one can increase the \emph{covariance} between $f( X_{\theta +h} (T))$ and $ f( X_{\theta } (T)  )$, thereby reducing the variance of $s_{\theta,h}(f,T) $ and making the estimation procedure more efficient. As mentioned in the introduction, three such couplings suggested in the existing literature are : Common Reaction Numbers (CRN) (see \cite{KSR1}), Common Reaction Paths (CRP) (see \cite{KSR1}) and Coupled finite-differences (CFD) (see \cite{DA}). Among these we only consider the two best performing schemes, CRP and CFD, in this paper.

In CRP the processes $X_\theta$ and $X_{\theta+h}$ are coupled by their random time-change representations according to 
\begin{align*}
X_\theta(t) = x_0 + \sum_{k =1}^K Y_k \left(  \int_{0}^t \lambda_k( X_\theta(s) ,\theta )  ds\right) \zeta_k  \  \textnormal{  and  } \
X_{\theta +h }(t) = x_0 + \sum_{k =1}^K Y_k \left(  \int_{0}^t \lambda_k( X_{\theta + h}(s) ,\theta +h )  ds\right) \zeta_k,
\end{align*}
where $\{Y_k : k = 1, . . . ,K\}$ is a family of independent unit rate Poisson processes. Note that these Poisson processes are same for both processes $X_\theta$ and $X_{\theta+h}$, indicating that their reaction paths are the same.

In CFD the processes $X_\theta$ and $X_{\theta+h}$ are coupled by their random time-change representations according to 
 \begin{align*}
X_\theta(t) &= x_0 + \sum_{k =1}^K Y_k \left(  \int_{0}^t \lambda_k( X_\theta(s) ,\theta ) \wedge \lambda_k( X_{\theta + h}(s) ,\theta +h )    ds\right) \zeta_k \\
& + \sum_{k =1}^K Y^{(1)}_k \left(  \int_{0}^t \left[ \lambda_k( X_\theta(s) ,\theta ) - \lambda_k( X_\theta(s) ,\theta ) \wedge \lambda_k( X_{\theta + h}(s) ,\theta +h )  \right]  ds\right) \zeta_k \\
\ \textnormal{ and } \ X_{\theta +h }(t) &= x_0 + \sum_{k =1}^K Y_k \left( \int_{0}^t \lambda_k( X_\theta(s) ,\theta ) \wedge \lambda_k( X_{\theta + h}(s) ,\theta +h )    ds\right) \zeta_k \\
& + \sum_{k =1}^K Y^{(2)}_k \left(  \int_{0}^t \left[ \lambda_k( X_{\theta+h}(s) ,\theta+h ) - \lambda_k( X_\theta(s) ,\theta ) \wedge \lambda_k( X_{\theta + h}(s) ,\theta +h )  \right]  ds\right) \zeta_k ,
\end{align*}
where $a \wedge b= \min\{a,b\}$ and $\{Y_k, Y^{(1)}_k, Y^{(2)}_k: k = 1, . . . ,K\}$ is again a family of independent unit rate Poisson processes. Under this coupling, the processes $X_\theta$ and $X_{\theta+h}$ have the same state until the first time the counting process corresponding to $Y^{(1)}_k$ or $Y^{(2)}_k$ fires for some $k$. The probability that this \emph{separation} time will come before $T$ is proportional to $h$, which is usually a small number. Hence for the majority of simulation runs, the processes $X_\theta$ and $X_{\theta+h}$ are together for the whole time interval $[0,T]$, suggesting that they are \emph{strongly coupled}. 

Note that both CRP and CRN are estimating the same quantity $S_{\theta,h}(f,T)$ \eqref{fd:form} and hence they both suffer from the same bias $|S_{\theta,h}(f,T) - S_{\theta} (f,T) |$. The only difference between these two methods is in the coupling of the processes $X_\theta$ and $X_{\theta+h}$, which alters the variance of $s_{\theta,h}(f,T)$. For a given example, the method with a lower variance will be more efficient as it will require lesser number of samples $(N)$ to achieve the desired statistical accuracy. 

For any finite-difference scheme, one can show that the bias $|S_{\theta,h}(f,T) - S_{\theta} (f,T) |$ is proportional to $h$ while $\textnormal{Var}( s_{\theta,h}(f,T))$ is proportional to $1/h$. Therefore by making $h$ smaller we may reduce the bias but we will have to pay a computational cost by generating a large number of samples for the required estimation. As mentioned in the introduction, this trade-off between bias and computational efficiency is the major drawback of finite-difference schemes as one generally does not know the \emph{right} value of $h$ for which $S_{\theta,h}(f,T) $ is \emph{close enough} to $S_{\theta} (f,T) $, while at the same time the variance of $\textnormal{Var}( s_{\theta,h}(f,T))$ is \emph{not too large}. 


\subsection{Unbiased schemes}\label{subsec:unbsdschemes}

Unbiased schemes are desirable for sensitivity estimation because one does not have to worry about the bias and the accuracy of the estimate can be improved by simply increasing the number of samples. The Girsanov method is the first such unbiased scheme (see \cite{Gir} and \cite{PW}). Recall the random time-change representation \eqref{main:rtcrep} of the process $( X_\theta (t) )_{t \geq 0}$. In the Girsanov method, we express $S_{\theta}(f,T)$ as \eqref{sens:exact} where
\begin{align}
\label{rvgir}
s_{\theta}(f,T) = f(X_\theta(T)) \sum_{k=1}^K \left( \int_{0}^{T} \frac{1}{ \lambda_k( X_\theta(t) ,\theta  ) } \frac{  \partial \lambda_k( X_\theta(t) ,\theta  ) }{  \partial \theta }    R_k(dt) - \int_{0}^{T} \frac{ \partial \lambda_k( X_\theta(t) ,\theta )  }{ \partial \theta } dt\right)
\end{align}
and $(R_k(t))_{t \geq 0}$ is the counting process given by
\begin{align*}
R_k(t) = Y_k\left( \int_{0}^{t}  \lambda_k( X_\theta(s) ,\theta  ) ds \right).
\end{align*}
In \eqref{rvgir}, $R_k(dt) = 1$ if and only if the $k$-th reaction fires at time $t$. For any other $t$, $R_k(dt) = 0$. Hence if the $k$-th reaction fires at times $t^{(k)}_1,\dots, t^{(k)}_{n_k}$ in the time interval $[0,T]$, then we can write
\begin{align*}
\int_{0}^{T} \frac{1}{ \lambda_k( X_\theta(t) ,\theta  ) } \frac{  \partial \lambda_k( X_\theta(t) ,\theta  ) }{  \partial \theta }    R_k(dt) = 
\sum_{j=1}^{n_k } \frac{1}{ \lambda_k( X_\theta(t^{(k)}_j) ,\theta  ) } \frac{  \partial \lambda_k( X_\theta(t^{(k)}_j) ,\theta  ) }{  \partial \theta }.
\end{align*}

The Girsanov method is simple to implement because realizations of the random variable $s_\theta(f,T)$ can be easily generated by simulating the paths of the process $X_\theta$ until time $T$. However, as mentioned before, the variance of $s_\theta(f,T)$ can be very high (see \cite{KSR1,KSR2,DA}), which is a serious problem as it implies that a large number of samples are needed to produce statistically accurate estimates. Since simulations of the process $X_\theta$ can be quite cumbersome, generating a large sample can take an enormous amount of time. 

Now consider the common situation where we have \emph{mass-action kinetics} and $\theta$ is the rate constant of reaction $k_0$. In this case, $\lambda_{k_0}$ has the form $\lambda_{k_0}(x,\theta) = \theta \lambda'_{k_0}(x)$ and for every $k \neq k_0$, $\lambda_{k}$ does not depend on $\theta$. Therefore 
\begin{align*}
\frac{  \partial \lambda_k( x,\theta  ) }{  \partial \theta }   =   \left\{
\begin{array}{cc}
\frac{1}{\theta}\lambda_{k_0}(x,\theta)& \textnormal{ if } k = k_0  \\
0 & \textnormal{ otherwise}
\end{array}
\right\}
\end{align*}
and hence the formula \eqref{rvgir} for $s_{\theta}(f,T) $ simplifies to
\begin{align}
\label{rvgir2}
s_{\theta}(f,T) = \frac{1}{\theta} f(X_\theta(T)) \left( N^{k_0}_\theta(T)   - \int_{0}^{T} \lambda_{k_0}( X_\theta(t) ,\theta )dt\right),
\end{align}
where $N^{k_0}_\theta(T)$ is the number of times reaction $k_0$ fires in the time interval $[0,T]$. This formula clearly shows that the Girsanov method cannot be used to estimate the sensitivity value at $\theta = 0$ even though $S_\theta(f,T)$ \eqref{def:sensitivity} is well-defined. This is a big disadvantage since the sensitivity value at $\theta = 0$ is useful for understanding network design as it informs us whether the presence or absence of reaction $k_0$ influences the output or not. Unfortunately the problem with the Girsanov method is not just limited to $\theta =0$. Even for $\theta$ close to $0$, the variance of $s_{\theta}(f,T)$ can be very high, rendering this method practically infeasible (see \cite{Our}). This is again a serious drawback as reaction rate constants with small values are frequently encountered in systems biology and other areas.

These issues with the Girsanov method severely restrict its use and also highlight the need for new unbiased schemes for sensitivity analysis. 
In our recent paper \cite{Our}, we provide a new unbiased scheme based on another sensitivity formula of the form \eqref{sens:exact}. To motivate this formula we first discuss the problem of computing parameter sensitivity in the deterministic setting.

Consider the deterministic model corresponding to the parameter-dependent reaction network described in the introduction. If $x_\theta(t) \in [0,\infty)^d$ denotes the vector of species \emph{concentrations} at time $t$, then the process $(x_\theta(t))_{t \geq 0}$ is the solution of the following system of ordinary differential equations 
\begin{align}
\label{detRRE}
\frac{dx}{dt} = \sum_{k=1}^K \hat{\lambda}_k(x, \theta) \zeta_k
\end{align}
with initial condition $x(0) =x_0$. Here $\hat{\lambda}_k(x,\theta)$ is the parameter dependent reaction flux (see \cite{Goutsias}) of the $k$-th reaction. We assume that each $\hat{\lambda}_k$ is a differentiable function of its arguments. Pick an observation time $T>0$ and a differentiable output function $f :\R^d \to \R$, and suppose we are interested in computing the sensitivity value
\begin{align*}
\frac{ \partial f(x_\theta(T))}{\partial \theta}. 
\end{align*}
Using \eqref{detRRE} we can write
\begin{align*}
f(x_\theta(T)) = f(x_0) +\sum_{k=1}^K   \int_{0}^T \hat{\lambda}_k(x_\theta(t), \theta) (\nabla f(x_\theta(t))\bullet \zeta_k) dt,
\end{align*}
where $\nabla f(x)$ denotes the gradient of the map $x \mapsto f(x)$ and $\bullet$ denotes the standard inner product on $\R^d$.
Differentiating the last equation with respect to $\theta$ gives us
\begin{align}
\label{deterministicparamsens}
\frac{ \partial f(x_\theta(T))}{\partial \theta} = \sum_{k=1}^K   \left( \int_{0}^T \frac{ \partial \hat{\lambda}_k(x_\theta(t), \theta) }{\partial \theta} (\nabla f(x_\theta(t))\bullet \zeta_k)  dt 
+\int_{0}^T   \nabla \left( \hat{\lambda}_k( x_\theta(t), \theta)   (\nabla f(x_\theta(t))\bullet \zeta_k)  \right)\bullet y_\theta(t)   dt \right), 
\end{align}
where $y_\theta(t) = \partial x_\theta(t) / \partial \theta$. Note that the values of $y_\theta(t)$ can be easily computed by solving the ordinary differential equation obtained by differentiating \eqref{detRRE} with respect to $\theta$. Hence in the deterministic setting, computation of the parameter sensitivity is relatively straightforward.

Observe that relation \eqref{deterministicparamsens} helps us in viewing parameter sensitivity as the sum of two parts. The first part measures the contribution due to the sensitivity of the reaction fluxes ($\hat{\lambda}_k$) while the second part measures the contribution due to the sensitivity of the states $x_\theta(t)$ for $t \geq 0$. The main result in \cite{Our} provides such a decomposition for parameter sensitivity in the stochastic setting. However, unlike the deterministic scenario, measuring the second contribution is computationally very challenging. In order to present this result, we need to define certain quantities. Let $(X_\theta(t))_{t \geq 0}$ be a Markov process with generator $\mathbb{A}_\theta$. For any $f: \N^d_0 \to \R$, $t \geq 0$ and $x \in \N^d_0$ define
\begin{align*}
\Psi_{\theta}(x,f,t) = \mathbb{E} \left(  f(X_\theta(t)) \middle\vert X_{\theta}(0) = x\right),
\end{align*}
and for any $k=1,\dots,K$ let
\begin{align}
\label{defdtheta}
D_{\theta}(x,f,t,k) & = \Psi_{\theta}(x + \zeta_k,f,t)  -   \Psi_{\theta}(x,f,t).
\end{align}
Define $\lambda_0(x,\theta) = \sum_{k=1}^K \lambda_k(x,\theta)$ and let $\sigma_0,\sigma_1,\sigma_2$ denote the successive jump times\footnote{We define $\sigma_0 = 0$ for convenience} of the process $X_\theta$. Theorem 2.3 in \cite{Our} shows that $S_\theta(f,T)$ satisfies \eqref{sens:exact} with
\begin{align}
\label{expr:stheta}
s_\theta(f,T) =  \sum_{k = 1}^K \left( \int_{0}^T \frac{ \partial \lambda_k ( X_\theta (t) , \theta ) }{ \partial \theta   } \Delta_{\zeta_k} f(X_\theta(t)) dt 
+ \sum_{  i = 0 : \sigma_i < T   }^{\infty}  \frac{  \partial  \lambda_k ( X_\theta(\sigma_{i}) ,\theta ) }{ \partial \theta}  R_{\theta}( X_\theta( \sigma_i) ,f, T -\sigma_i ,k) \right), 
\end{align} 
where
\begin{align}
\label{defn_rtheta}
R_{\theta}(x,f,t,k) =  \int_{0}^{t} \left(  D_{\theta}(x,f,s,k)   - \Delta_{\zeta_k} f(x) \right) e^{ - \lambda_0(x,\theta) (t-s)} ds.
\end{align} 

This result was proved by coupling the processes $X_\theta$ and $X_{\theta+h}$ as in CFD (see Section \ref{subsec:fdschemes}) and computing the limit of the finite-difference $S_{\theta,h}(f,T)$ (see \eqref{fd:form}) as $h\to 0$. Similar to the deterministic case, this result shows that parameter sensitivity $S_{\theta,h}(f,T)$ can be seen as the sum of two parts : the first part measures the contribution due to the sensitivity of the propensity functions ($\lambda_k$) while the second part measures the contribution due to the sensitivity of the states $X_\theta(\sigma_i)$ of the Markov process at the jump times $\sigma_i$. Computing the latter contribution is difficult because it involves the function $R_\theta$ which generally does not have an explicit formula. To overcome this problem, one needs to estimate all the quantities of the form $R_{\theta}(x,f,t,k) $ that appear in \eqref{expr:stheta}. However the number of these quantities is proportional to the number of jumps of the process before time $T$, which can be quite high even for small networks. If we estimate each of these quantities \emph{serially}, as and when they appear, using a collection of independently generated paths of the process $X_\theta$, then the problem becomes computationally intractable for most examples of interest. In \cite{Our} we devised a scheme, called the \emph{Auxiliary Path Algorithm}(APA), to estimate all these quantities in \emph{parallel} by generating a fixed number of \emph{auxiliary} paths in addition to the main path.  APA stores information about all the required quantities in a big \emph{Hashtable} and tracks the states visited by the auxiliary paths to estimate those quantities. Due to all the necessary book-keeping, APA is hard to implement and it also has high memory requirements. In fact the space and time complexity of APA scales linearly with the number of jumps in the time interval $[0,T]$ and this makes APA practically unusable for large networks or for large values of observation time $T$. Moreover APA only works well if the stochastic dynamics visits the same states again and again, which may not be generally true.

The above discussion suggests that if we can modify the expression for $s_\theta(f,T)$ in such a way that only a small fraction of unknown quantities (of the form $R_{\theta}(x,f,t,k) $) require estimation, then it can lead to an efficient method for sensitivity estimation. We exploit this idea in this paper. By adding \emph{extra randomness} to the random variable $s_\theta(f,T)$, we construct another random variable $\hat{s}_\theta(f,T)$, which has the same expectation as $s_\theta(f,T)$
\begin{align}
\label{expispreserved}
\E\left( s_\theta(f,T)  \right) =  \E\left(  \hat{s}_\theta(f,T) \right),
\end{align}
even though its distribution may be different.
We then show that realizations of the random variable $\hat{s}_\theta(f,T)$ can be easily obtained through a simple procedure. This gives us a new method for unbiased parameter sensitivity estimation for stochastic reaction networks.

\subsection{Construction of $\hat{s}_\theta(f,T) $} \label{sec:constrcstheta}

We now describe the construction of the random variable $\hat{s}_\theta(f,T) $ introduced in the previous section. Let $(X_\theta (t))_{t \geq 0}$ be the Markov process with generator $\mathbb{A}_{\theta}$ and initial state $x_0$. Let $\sigma_0,\sigma_1,\dots$ denote the successive jump times of this process.
The total number of jumps until time $T$ is given by the random variable
\begin{align}
\label{defn_eta}
\eta = \max \{ i \geq 0 :  \sigma_i < T \}.
\end{align}
Note that for any $T>0$ we have $\eta \geq 1$ because $\sigma_0 = 0$. If the Markov process $X_\theta$ reaches a state $x$ for which $\lambda_0(x,\theta) = 0$, then $x$ is an \emph{absorbing} state and the process stays in $x$ forever. From \eqref{defn_eta}, it is immediate that for any non-negative integer $i < \eta$, $X_\theta( \sigma_i )$ cannot be an absorbing state. Let $\alpha$ indicate if the final state $X_\theta( \sigma_\eta )$ is absorbing
\begin{align*}
\alpha= \left\{
\begin{tabular}{cc}
1 & \textnormal{ if } $\lambda_0( X_\theta( \sigma_\eta ) , \theta) =0$ \\
0 & \textnormal{ otherwise}
\end{tabular} \right\}.
\end{align*}
For each $i=0,\dots,(\eta-\alpha)$ let $\gamma_i$ be an independent exponentially distributed random variable with rate $\lambda_0( X_\theta( \sigma_i ),\theta )$ and define
\begin{align*}
\Gamma_i =   \left\{
\begin{array}{cc}
1 & \textnormal{ if } \gamma_i < (T - \sigma_i)  \\
0 & \textnormal{ otherwise}
\end{array}
\right\}. 
\end{align*}

For each $i=0,\dots,(\eta-\alpha)$ and $k=1,\dots,K$ let $ \beta_{ki} $ be given by
\begin{align*}
 \beta_{ki} = \textnormal{Sign}\left(  \frac{ \partial \lambda_k( X_\theta( \sigma_i) ,\theta ) }{  \partial \theta }   \right) \  \textnormal{ where } \ \textnormal{Sign}(x)= \left\{
\begin{array}{cc}
1   &   \textnormal{ if }  x > 0  \\
-1   &  \textnormal{ if }   x < 0  \\
0 &    \textnormal{ if } x= 0
\end{array}
\right\}.
\end{align*}
Fix a normalizing constant $c>0$. The choice of $c$ and its role  will be explained later in the section. If $\beta_{ki} \neq 0$ and $\Gamma_i = 1$, then let $\rho_{ki}^{(c)}$ be an independent $\N_0$-valued random variable whose distribution is Poisson with parameter
\begin{align}
\label{rateofpoisoon}
\frac{c}{ \lambda_0( X_\theta(\sigma_i) ,\theta ) }  \left|  \frac{ \partial \lambda_k (X_\theta(\sigma_i) ,\theta ) }{ \partial \theta }  \right|.
\end{align}
Here the denominator $ \lambda_0( X_\theta(\sigma_i) ,\theta ) $ is non-zero because for $i \leq (\eta-\alpha)$ the state $X_\theta(\sigma_i)$ cannot be absorbing. On the event $\{ \beta_{ki} \neq 0,  \Gamma_i = 1 \textnormal{ and } \rho_{ki}^{(c)} >0\}$ we define another random variable as 
\begin{align}
\label{evaluatedki}
\hat{D}_{ki}=  f(Z_1(T - \sigma_i -\gamma_i )) - f(Z_2(T - \sigma_i- \gamma_i)),
\end{align}
where $(Z_1(t) )_{t \geq 0}$ and $(Z_2(t) )_{t \geq 0}$ are two processes which are coupled by the following random time-change representations:
\begin{align*}
Z_1(t) &=( X_\theta(\sigma_i)+ \zeta_k) + \sum_{k = 1}^K \hat{Y}_k\left( \int_{0}^{t}  \lambda_k( Z_1(s) ,\theta) \wedge \lambda_k( Z_2(s) ,\theta) ds \right)\zeta_k \\
&+  \sum_{k = 1}^K \hat{Y}^{(1)}_k\left( \int_{0}^{t} \left(\lambda_k( Z_1(s) ,\theta) -  \lambda_k( Z_1(s) ,\theta) \wedge \lambda_k( Z_2(s) ,\theta) \right) ds \right) \zeta_k \\  \textnormal{ and } \  
Z_2(t) &=X_\theta(\sigma_i)+ \sum_{k = 1}^K \hat{Y}_k\left( \int_{0}^{t}  \lambda_k( Z_1(s) ,\theta) \wedge \lambda_k( Z_2(s) ,\theta) ds \right)\zeta_k \\
&+  \sum_{k = 1}^K \hat{Y}^{(2)}_k\left( \int_{0}^{t} \left(\lambda_k( Z_2(s) ,\theta) -  \lambda_k( Z_1(s) ,\theta) \wedge \lambda_k( Z_2(s) ,\theta) \right) ds \right) \zeta_k,
\end{align*}
where $\{\hat{Y}_k, \hat{Y}^{(1)}_k,\hat{Y}^{(2)}_k : k =1,\dots,K\}$ is an independent family of unit rate Poisson processes. This coupling is similar to the coupling used by CFD (see Section \ref{subsec:fdschemes}). Note that  $(Z_1(t) )_{t \geq 0}$ and $(Z_2(t) )_{t \geq 0}$ are Markov processes with generator $\mathbb{A}_\theta$ and initial states $( X_\theta(\sigma_i)+ \zeta_k) $ and $X_\theta(\sigma_i)$ respectively. Therefore 
\begin{align}
\label{diff_estimation}
\E\left( \hat{D}_{ki}  \middle\vert  T - \sigma_i -\gamma_i , X_\theta(\sigma_i)  \right) & = \Psi_{\theta}(X_\theta(\sigma_i) +\zeta_k,f, T - \sigma_i -\gamma_i)- \Psi_{\theta}(X_\theta(\sigma_i) ,f, T - \sigma_i -\gamma_i) \notag \\
&= D_\theta(X_\theta(\sigma_i) ,f,T - \sigma_i -\gamma_i,k),
\end{align}
where $D_\theta$ is defined by \eqref{defdtheta}.
In other words, given $T - \sigma_i -\gamma_i =t$ and $X_\theta(\sigma_i) = x$, the mean of the random variable $\hat{D}_{ki}$ is just $D_\theta(x,f,t,k)$. The above coupling between $(Z_1(t) )_{t \geq 0}$ and $(Z_2(t) )_{t \geq 0}$ makes them strongly correlated, thereby lowering the variance of the difference $\hat{D}_{ki} = f(Z_1(t)) -f(Z_2(t))$. This strong correlation is evident from the fact that if $Z_1(s) = Z_2(s)$ for some $s\geq0$ then $Z_1(u) = Z_2(u)$ for all $u\geq s$. 
Finally, if the last state $X_\theta( \sigma_\eta )$ is absorbing (that is, $\alpha=1$) and $\partial \lambda_k( X_\theta( \sigma_\eta ) ,\theta )/ \partial \theta$ is non-zero then we define another random variable $\hat{I}_{k \eta}$ as
\begin{align}
\label{evaluateiketa}
\hat{I}_{k \eta} = \int_{0}^{T - \sigma_\eta} f(Z(s))ds,
\end{align}
where $(Z(t))_{t \geq 0}$ is an independent Markov process with generator $\mathbb{A}_\theta$ and initial state $X_\theta(\sigma_\eta)+\zeta_k$.
Note that
\begin{align}
\label{int_estimation}
\E\left( \hat{I}_{k \eta} \vert X_\theta(\sigma_\eta), T - \sigma_\eta \right) = \int_{0}^{T - \sigma_\eta} \Psi_\theta( X_\theta(\sigma_\eta) +\zeta_k, f, s)ds.
\end{align}

We are now ready to provide an expression for $\hat{s}_\theta(f,T)$. Let
\begin{align*}
\Delta t_i  = \left\{
\begin{array}{cc}
 (\sigma_{i+1} - \sigma_i) & \textrm{ for } i = 0,\dots,\eta-1  \\
 (T - \sigma_\eta) & \textrm{ for } i = \eta  \\
\end{array}\right\}
\end{align*}
and define
\begin{align}
\label{expr:sthetahat}
\hat{s}_\theta(f,T) & = \sum_{k=1}^K \sum_{i=0}^{\eta - \alpha } \left[   \frac{ \partial \lambda_k ( X_\theta (  \sigma_i  ) , \theta ) }{ \partial \theta   } \Delta_{\zeta_k} f(X_\theta(  \sigma_i  )) \left( \Delta t_i - \frac{  \Gamma_i }{ \lambda_0( X_\theta (  \sigma_i  )  ,\theta ) }  \right) + \frac{1}{c}  \beta_{ki}  \Gamma_i \rho^{(c)}_{ki} \hat{D}_{ki} \right] \\
& + \alpha \sum_{k=1}^K  \left[   \frac{ \partial \lambda_k ( X_\theta (  \sigma_\eta  ) , \theta ) }{ \partial \theta   } \left( \hat{I}_{ k \eta }   - (T - \sigma_\eta) f ( X_\theta (  \sigma_\eta  ) ) \right)\right]  .  \notag
\end{align}
By a simple conditioning argument we prove in Section \ref{asec:proof} of \emph{Supplementary Materials} that relation \eqref{expispreserved} holds.

We mentioned before that it is difficult to obtain realizations of $s_\theta(f,T)$ because one needs to estimate a quantity like $R_{\theta}(x,f,t,k)$ at each jump time $\sigma_i < T$, that requires simulation of new paths of the process $X_\theta$ which is computationally expensive. Similarly for $\hat{s}_\theta(f,T)$ we need to compute $\hat{D}_{ki}$ at each $\sigma_i < T$ which also requires simulation of new paths of the process $X_\theta$. However the main difference is that to compute $\hat{s}_\theta(f,T)$, $\hat{D}_{ki}$ is only needed if the Poisson random variable $\rho^{(c)}_{ki}$ is strictly positive. If we can effectively control the number of positive $\rho^{(c)}_{ki}$-s then we can efficiently generate realizations of $\hat{s}_\theta(f,T)$. We later explain how this control can be achieved using the positive parameter $c$ introduced in the definition of the random variable $\rho^{(c)}_{ki}$ (see \eqref{rateofpoisoon}). The construction of $\hat{s}_\theta(f,T)$ outlined above, also provides a recipe for obtaining realizations of this random variable and hence gives us a method for estimating the parameter sensitivity $S_\theta(f,T)$. We call this method, the \emph{Poisson Path Algorithm} (PPA) because at each jump time $\sigma_i$, the crucial decision of whether new paths of the process $X_\theta$ are needed (for $\hat{D}_{ki}$) is based on the value of a Poisson random variable $\rho^{(c)}_{ki}$. We describe PPA in greater detail in the next section.

We now discuss how the positive parameter $c$ can be selected.
Let $\eta_{ \textnormal{req}}$ denote the total number of positive $\rho^{(c)}_{ki}$-s that appear in \eqref{expr:sthetahat}. This is the number of $\hat{D}_{ki}$-s that are required to obtain a realization of $\hat{s}_\theta(f,T)$. It is immediate that $\eta_{ \textnormal{req}}$ is bounded above by $\rho_{ \textnormal{tot} } = \sum_{k=1}^K \sum_{i=0}^{\eta -\alpha} \rho^{(c)}_{ki}$, which is a Poisson random variable with parameter $c R_{\textnormal{tot}} $, where
\begin{align}
\label{defnrtot}
R_{\textnormal{tot}} = \sum_{k=1}^K \sum_{i=0}^{\eta -\alpha}  \frac{1}{ \lambda_0( X_\theta(\sigma_i) ,\theta ) }  \left|  \frac{ \partial \lambda_k (X_\theta(\sigma_i) ,\theta ) }{ \partial \theta }  \right|.
\end{align}
By picking a small $c>0$ we can ensure that $\E\left( \rho_{ \textnormal{tot} } \right) = c \E\left(  R_{\textnormal{tot}} \right)$ is small, which would also guarantee that $\rho_{ \textnormal{tot} }$ and $\eta_{ \textnormal{req}}$ are small. Specifically we choose a small positive integer $M_0$ (like $10$, for instance) and set 
\begin{align}
\label{normconstant}
c = \frac{M_0}{\E\left(  R_{\textnormal{tot}} \right)},
\end{align}
where $\E\left(  R_{\textnormal{tot}} \right)$ is estimated using $N_0$ simulations of the process $X_\theta$ in the time interval $[0,T]$. The choice of $N_0$ is not critical and typically a small value (like $100$, for example) is sufficient to provide a decent estimate of $\E\left(  R_{\textnormal{tot}} \right)$. The role of parameter $M_0$ is also not very important in determining the efficiency of PPA. If $M_0$ increases then $\eta_{ \textnormal{req}}$ increases as well, and the computational cost of generating each realization of $\hat{s}_\theta(f,T)$ becomes higher. However as $\eta_{ \textnormal{req}}$ increases the variance of $\hat{s}_\theta(f,T)$ decreases and hence fewer realizations of $\hat{s}_\theta(f,T)$ are required to achieve the desired statistical accuracy. These two effects usually offset each other and the overall efficiency of PPA remains relatively unaffected. Note that PPA provides unbiased estimates for the sensitivity values, regardless of the choice of $N_0$ and $M_0$.

\section{The Poisson Path Algorithm (PPA)} \label{sec:ppa}

We now describe PPA which is essentially a method for obtaining realizations of the random variable $\hat{s}_{\theta}(f,T)$ (defined by \eqref{expr:sthetahat}) for some observation time $T$ and some output function $f$. Since $S_{\theta}(f,T) = \E\left( \hat{s}_{\theta}(f,T)\right)$, we can estimate $S_{\theta}(f,T)$ by generating $N$ realizations $s^{(1)}_{\theta}(f,T),\dots,s^{(N)}_{\theta}(f,T)$ of the random variable $\hat{s}_\theta(f,T)$ and then computing their empirical mean \eqref{sens_est1}.

Let $x_0$ denote the initial state of the process $X_\theta$ and assume that the function \emph{rand()} returns independent samples from the uniform distribution on $[0,1]$. We simulate the paths of our Markov process using Gillespie's SSA \cite{GP}. When the state of the process is $x$, the next time increment ($\Delta t$) and reaction index ($k$) is given by the function $\Call{SSA}{x}$ (see Algorithm \ref{SSA} in Section \ref{app:methods} of \emph{Supplementary Materials}). Before we can generate $\hat{s}_{\theta}(f,T)$ we need to fix a normalization parameter $c$ according to \eqref{normconstant}. 
For this we estimate $\E\left(  R_{\textnormal{tot}} \right)$ using $N_0$ simulations of the Markov process (see Algorithm \ref{estimatenormalization} in Section \ref{app:methods} of \emph{Supplementary Materials}).

Once $c$ is chosen, a single realization of the random variable $\hat{s}_\theta(f,T)$ (given by \eqref{expr:sthetahat}) can be computed using $\Call{GenerateSample}{x_0,T,c}$ (Algorithm \ref{gensensvalue}). This method simulates the process $X_{\theta}$ according to SSA and at each state $x$ and jump time $t$, the following happens:
\begin{itemize}
\item If $x$ is an absorbing state ($\lambda_0(x,\theta)=0$) then $t$ is the last jump time before $T$ ($t=\sigma_\eta$). For each $k=1,\dots,K$ such that $\partial \lambda_k(x,\theta)/ \partial \theta \neq 0$, the quantity $\hat{I}_{k \eta}$ (see \eqref{evaluateiketa}) is evaluated using $\Call{EvaluateIntegral}{x +\zeta_k,T-t}$ (see Algorithm \ref{evaluateintegralindependent} in Section \ref{app:methods} of \emph{Supplementary Materials}) and then used to update the sample value according to \eqref{expr:sthetahat}.

\item If $x$ is \emph{not} an absorbing state, then $t= \sigma_i$ for some jump time $\sigma_i$ with $i<\eta$. The exponential random variable $\gamma$ (where $\gamma = \gamma_i$ in \eqref{expr:sthetahat}) is generated and for each $k=1,\dots,K$ such that $\partial \lambda_k(x,\theta)/ \partial \theta \neq 0$, the Poisson random variable $n$ (where $n = \rho^{(c)}_{ki}$ in \eqref{expr:sthetahat}) is also generated.
If $\gamma < (T-t)$ and $n>0$ then the quantity $\hat{D}_{ki}$ (see \eqref{evaluatedki}) is evaluated using $\Call{EvaluateCoupledDifference}{x,x+\zeta_k,T-t-\gamma}$ (see Algorithm \ref{gensensvalue} in Section \ref{app:methods} of \emph{Supplementary Materials}) and then used to update the sample value according to \eqref{expr:sthetahat}. To generate a Poisson random variable with parameter $r$ we use the function $\Call{GeneratePoisson } { r }$ (see Algorithm \ref{genpoissrv} in Section \ref{app:methods} of \emph{Supplementary Materials}).
\end{itemize}

\begin{algorithm}[h]
\caption{Generates one realization of $\hat{s}_\theta(f,T)$ according to \eqref{expr:sthetahat} }
\label{gensensvalue}     
\begin{algorithmic}[1]
\Function{GenerateSample}{$x_0,T,c$}
    \State Set $x = x_0$, $t = 0$ and $s = 0$
\While {$ t <  T $} 
\State  Calculate $(\Delta t, k_0 ) =$ SSA$(x)$
\State Update $\Delta t  \gets \min\{\Delta t, T -t\}$
 \If {$\lambda_0(x,\theta) > 0$} 
 \State  Set $\gamma = -\frac{ \log(rand())}{ \lambda_0(x,\theta) } $
\EndIf
\For {$k = 1$ to $K$}	
\State Set $r = \left|  \frac{ \partial \lambda_k(x,\theta) }{ \partial \theta }  \right|$ and $\beta = \textnormal{Sign}\left(  \frac{ \partial \lambda_k(x,\theta) }{ \partial \theta }\right)$
\If {$r > 0$}
\If {$\lambda_0(x,\theta) = 0$}
\State  Update $s \gets s + \left( \frac{ \partial \lambda_k(x,\theta) }{ \partial \theta } \right) \left( \Call{EvaluateIntegral}{x +\zeta_k,T-t} - (T-t)f(x)\right)$
\Else 
\State Set $n = \Call{GeneratePoisson } { \frac {r c }{ \lambda_0(x,\theta) }  }$
\If{ $\gamma < (T - t)$ }
\State Update $s \gets s + \left( \frac{ \partial \lambda_k(x,\theta) }{ \partial \theta } \right) (f(x+\zeta_k) -f(x) ) \left( \Delta t - \frac{ 1 }{ \lambda_0(x,\theta) }  \right)$
\If {$n > 0$}
\State Update $s \gets s +\left( \frac{\beta n}{c}  \right) \Call{EvaluateCoupledDifference}{x,x+\zeta_k,T-t-\gamma} $
\EndIf
\Else
\State Update $s \gets s + \left( \frac{ \partial \lambda_k(x,\theta) }{ \partial \theta } \right) (f(x+\zeta_k) -f(x) ) \Delta t$
\EndIf
\EndIf
\EndIf
\EndFor
\State Update $t \gets t +\Delta t$ and $x \gets x+\zeta_{k_0}$
\EndWhile 
\State \Return $s$
\EndFunction
\end{algorithmic}
\end{algorithm}

\section{Numerical Examples} \label{sec:ex}

In this section, we present many examples to compare the performance of PPA with the other methods for sensitivity estimation. Among these other methods, we consider the Girsanov method which is unbiased, and the two best-performing finite-difference schemes, CFD and CRP, which are of course biased. We can compare the performance of different methods by comparing the time they need to produce a \emph{statistically accurate estimate} of the true sensitivity value $S_{\theta}(f,T)$ given by \eqref{def:sensitivity}. Our first task is to find a way to judge if an estimate of $S_\theta(f,T)$ is statistically accurate.

Suppose that a method estimates parameter sensitivity using $N$ samples whose mean is $\mu_N$ and standard deviation is $\sigma_N$ (see Section \ref{sec:pre}). Assume that the true value of $S_{\theta}(f,T)$ is $s_0$. Let $a = s_0 - 0.05|s_0|$ and $b =s_0 + 0.05|s_0|  $, where $|\cdot|$ denotes the absolute value function. Then $[a,b]$ is the $5\%$ interval around the true value $s_0$. Under the Central Limit approximation, the estimator $\hat{S}_\theta(f,T)$ can be seen as a normally distributed random variable with mean $\mu_N$ and variance $\sigma^2_N$. Hence using \eqref{probestimate} we can calculate the probability $$p = \P( \hat{S}_\theta(f,T) \in [a,b] )$$ that the estimator lies within the $5\%$ interval around $s_0$. Henceforth we refer to $p$ as the \emph{confidence level} for the estimator. Observe that the statistical accuracy of an estimate can be judged from the value of $p$: the higher the value of $p$, the more accurate is the estimate. By calculating the probability $p$ using the $5\%$ interval around $s_0$, we ensure that the statistical precision is high or low depending on whether the sensitivity value is small or large.

We mentioned in Section \ref{sec:pre} that the standard deviation $\sigma_N$ is inversely proportional to the number of samples $N$, which shows that $\sigma_N \to 0$ as $N \to \infty$. For an unbiased scheme (Girsanov and PPA), $\mu_N \to s_0$ as $N \to \infty$, due to the law of large numbers. Therefore from \eqref{probestimate} it follows that for an unbiased scheme, the confidence level $p$ can be made as close to $1$ as desired by picking a large sample size $N$. However the same is not true for finite-difference schemes such as CFD and CRP. In such schemes $\mu_N \to S_{\theta,h}(f,T)$ as $N \to \infty$, where $S_{\theta,h}(f,T)$ given by \eqref{fd:form} is generally different from the true value $S_{\theta}(f,T)=s_0$ because of the bias. If the bias is large, then $S_{\theta,h}(f,T)$ does not lie inside the $5\%$ interval $[a,b]$ around $s_0$, and in this case the confidence level $p$ is close to $0$ for large values of $N$. Therefore if high confidence levels are desired with finite-difference schemes, then a small $h$ must be chosen to ensure that the bias $|S_{\theta,h}(f,T) -S_{\theta}(f,T)|$ is low. However the variance $\sigma^2_N$ of the associated estimator scales like $1/h$ (see Section \ref{subsec:fdschemes}), implying that if $h$ is \emph{too small} then an extremely large sample size is required to make $\sigma_N$ sufficiently small in order to achieve high confidence levels. Generating a large sample may impose a heavy computational burden as simulations of the underlying Markov process can be quite cumbersome.

In this paper, we compare the efficiency of different methods by measuring the CPU time\footnote{All the computations in this paper were performed using C++ programs on a Linux machine with a 2.4 GHz Intel Xeon processor.} that is needed to produce a sample with least possible size, such that the corresponding estimate has a certain threshold confidence level. This threshold confidence level is \emph{close} to $1$ (either $0.95$ or $0.99$) and hence the produced estimate is statistically quite accurate. The discussion in the preceding paragraph suggests that for finite-difference schemes, there is a trade-off between the statistical accuracy and the computational burden. This trade-off is illustrated in Figure \ref{fig:tradeoff} which depicts $4$ cases that correspond to $4$ values of $h$ arranged in decreasing order. Hence $h$ in Case 1 is the largest while the $h$ in Case 4 is the smallest. We describe these cases below.
\begin{itemize}
\item[{\bf Case 1:}] Since $h$ is large, the computational burden is low but due to the large bias, the estimator distribution (normal curve in Figure \ref{fig:tradeoff}) puts very little mass on the $5\%$ interval around the true sensitivity value. Therefore the confidence level $p$ is low and the estimate is unacceptable. 
\item[{\bf Case 2:}] In comparison to Case 1, $h$ is smaller, the computational burden is higher and the bias is lower. However the estimator distribution still does not put enough mass on the $5\%$ interval and so the value of $p$ is not high enough for the estimate to be acceptable.
 \item[{\bf Case 3:}] In comparison to Case 2, $h$ is smaller and so the computational burden is higher, but fortunately the bias is small enough to ensure that the estimator distribution puts nearly all its mass on the $5\%$ interval. Therefore the confidence level $p$ is close to $1$ and the estimate is acceptable.
  \item[{\bf Case 4:}] Here $h$ is smallest and so the computational burden is the highest among all the cases. However in comparison to Case 3, the additional computational burden in unnecessary because there is only a slight improvement in the bias and the confidence level. 
\end{itemize}
\begin{figure}[ht]
  \begin{center}
    \includegraphics[height=10cm]{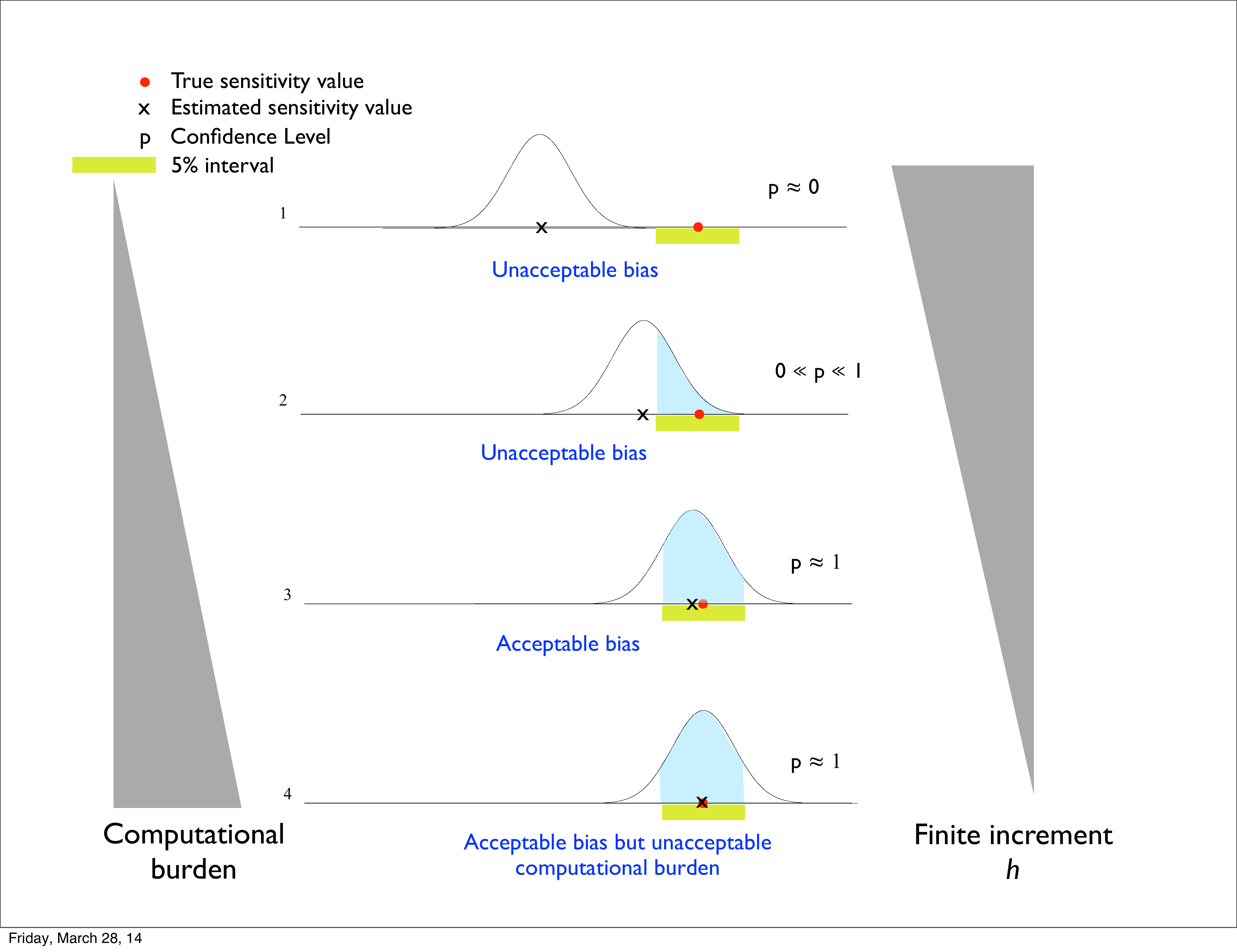}
  \caption{Trade-off between accuracy and computational burden for finite-difference schemes. The $4$ cases correspond to $4$ values of $h$ arranged in decreasing order. For Case 2, $0 \ll p \ll 1$ should be interpreted as ``$p$ is far away from $0$ and $1$".} \label{fig:tradeoff}
       \end{center}

\end{figure}

In order to use a finite-difference scheme efficiently one needs to use the largest value of $h$ for which an estimate with a high confidence level can be produced (Case 3 in Figure \ref{fig:tradeoff}). However, in general, it is impossible to determine this value of $h$. A blind choice of $h$ is likely to correspond to situations $1$ or $4$. To make matters worse, since the correct sensitivity value $S_\theta(f,T)$ is usually unknown one cannot measure the statistical accuracy of a finite-difference scheme, without re-estimating $S_\theta(f,T)$ using an unbiased method. Most practitioners who use finite-difference schemes, assume that the statistical accuracy is high if $h$ is small. However as our examples illustrate, an estimate can be inaccurate even for a $h$ that appears small. Therefore if high accuracy is needed for an application, then an unbiased method should be preferred over finite-difference schemes.

Since it is difficult to determine the \emph{optimum} value of $h$, we adopt the following strategy to gauge the efficiency of a finite-difference scheme. We start with $h=0.1$ and check if the threshold confidence level $p$ can be achieved for some sample size $N$. If it does, then we stop here, otherwise we decrease $h$ by a factor of $10$ (to $h=0.01$) and again perform the check. Continuing this way, we eventually arrive at a value of $h$ in the sequence $0.1,0.01,0.001\dots$, for which the threshold confidence level $p$ is achievable for a large enough sample size. For performance comparison, we \emph{only} use the last CPU time corresponding to the value of $h$ for which the check was successful. Note that we do not account for the time wasted in all the unsuccessful attempts. Therefore our strategy for performance comparison is more lenient towards the finite-difference schemes and more conservative towards the unbiased schemes. This leniency is intentionally adopted to compensate for the arbitrariness in choosing the ``test" values of $h$.

To compute the confidence level $p$ we need to know the exact sensitivity value $S_\theta(f,T) $. In some of our examples, $S_{\theta}(f,T)$ can be analytically evaluated and hence the computation of $\epsilon$ is quite straightforward. In other examples, where $S_{\theta}(f,T)$ cannot be explicitly computed, we use PPA to produce an estimate with a sample whose size is large enough to have very low sample variance. The rationale behind this is that since PPA is unbiased, if the sample variance is low then the estimate would be close to the true sensitivity value $S_\theta(f,T)$. 

We now start discussing the examples. In all the examples, the propensity functions $\lambda_k$ are in the form of mass-action kinetics unless stated otherwise. Recall that our method PPA depends on two parameters $N_0$ and $M_0$ whose choice is not very important (see Section \ref{sec:constrcstheta}). In this paper we always use PPA with $N_0=100$ and $M_0=10$. In all the examples we provide a bar chart with the CPU time required by each method to produce an estimate with the desired confidence level $p=0.95$ or $p=0.99$. This facilitates an easy performance comparison between various methods. The exact values of the CPU times, sample size $N$, estimator mean $\mu_N$, estimator standard derivation $\sigma_N$ and $h$ (for finite-difference schemes) are provided in Section \ref{sec:data} of \emph{Supplementary Materials}.

\begin{example}[Single-species birth-death model]
\label{ex:birth_death}
{ \rm 
Our first example is a simple birth-death model in which a single species $\mathcal{S}$ is created and destroyed according to the following two reactions:
\begin{align*}
\emptyset \stackrel{\theta_1 }{\rightarrow} \mathcal{S} \stackrel{\theta_2 }{\rightarrow} \emptyset.
\end{align*}
Let $\theta_1= \theta_2 = 0.1$ and assume that the sensitive parameter is $\theta = \theta_2$. Let $(X(t))_{t \geq 0}$ be the Markov process representing the reaction dynamics. Hence the population of $\mathcal{S}$ at time $t$ is given by $X(t) \in \N_0$. Assume that $X(0)=0$. For $f(x)=x$ we wish to estimate
\begin{align*}
S_\theta(f,T) = \frac{ \partial  }{  \partial \theta } \E\left( f( X(T) ) \right) = \frac{ \partial  }{  \partial \theta } \E\left(  X(T)  \right).
\end{align*}
for $T=20$ and $100$. Since the propensity functions of this network are affine, we can compute $S_\theta(f,T)$ exactly (see Section \ref{sec:data} of \emph{Supplementary Materials}). These exact values help in computing the confidence level of an estimate.

We estimate the sensitivity values with all the methods with two threshold confidence levels ($p=0.95$ and $p=0.99$), and plot the corresponding CPU times in Figure \ref{fig:bd}. For finite-difference schemes (CFD and CRP), the value of $h$ for which the desired confidence level was achieved is also shown.
\begin{figure}[h!]
  \begin{center}
    \includegraphics[height=10.0cm]{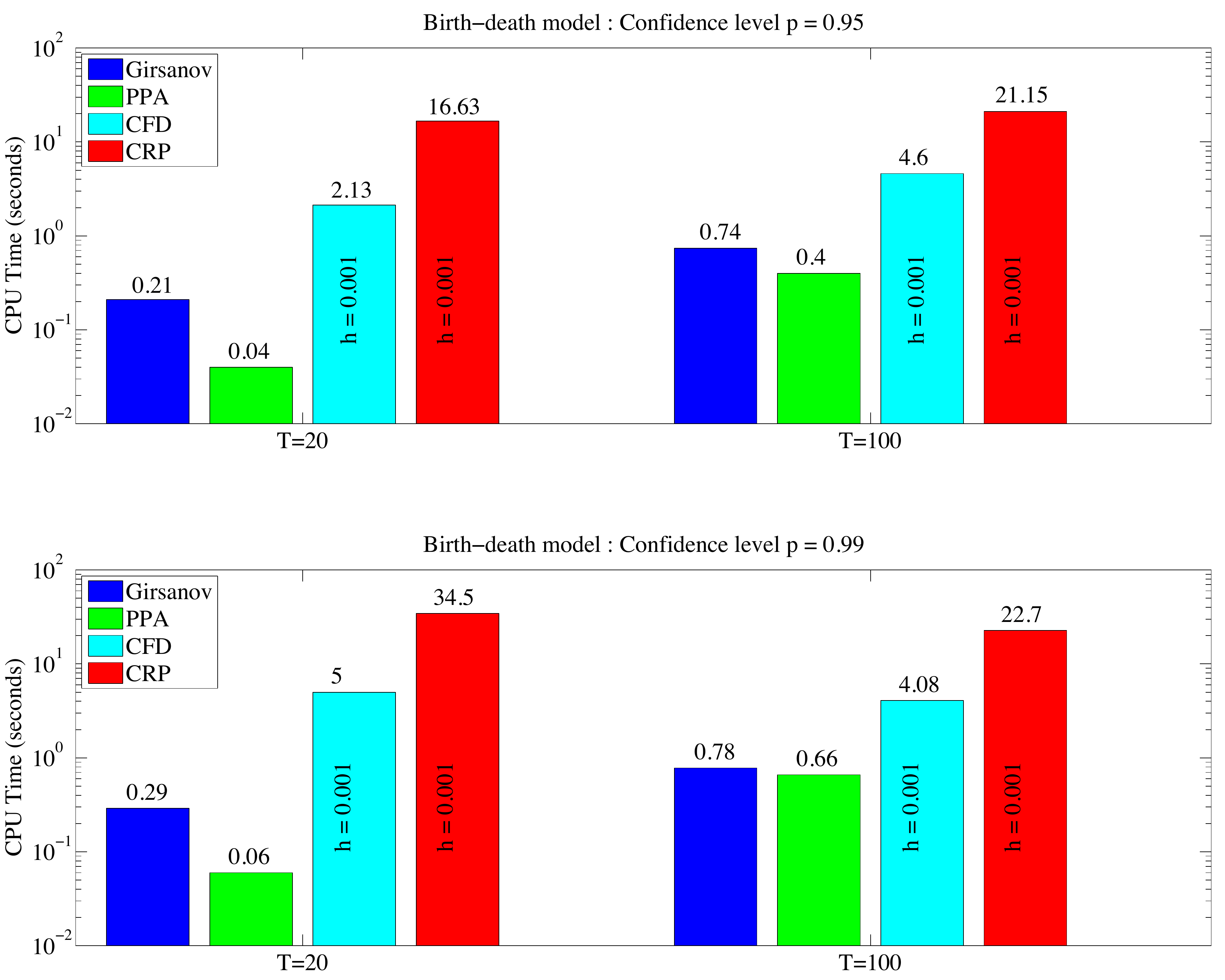}
  \end{center}
  \caption{Efficiency comparison for birth-death model with two threshold confidence levels : $p=0.95$ and $p=0.99$ and two observation times : $T=20$ and $T=100$. The exact CPU time is shown at the top of each bar. Note that efficiency is inversely related to the CPU time.}
\label{fig:bd}
\end{figure}

From Figure \ref{fig:bd}, it is immediate that for this example PPA generally performs better than the other methods. To measure this gain in performance more quantitatively, we compute the \emph{average speed-up factor} of PPA with respect to another method. For each threshold confidence level, this factor is calculated by simply taking the ratio of the aggregate CPU times required by a certain method and PPA, to perform all the sensitivity estimations. In this example, this aggregate involves the CPU times for $T=20$ and $T=100$. These average speed-up factors are presented in Table \ref{tb:speedup}. Note that in this example, PPA is significantly faster than the finite-difference schemes.

\begin{table}[h]
\caption{Birth-Death model - Average speed-up factors of PPA w.r.t other methods  } 
\label{tb:speedup} 
\begin{center}
\begin{tabular}{| c |c | c | c |}
\hline
 $p$ & Girsanov & CFD & CRP  \\ \hline 
  0.95 & 2.15 & 15.3& 85.9\\
  0.99 & 1.5 & 12.6 &  79.4\\ \hline
\end{tabular}
\end{center}
\end{table}

Finally we would like to use this example to demonstrate the pitfalls of choosing a ``wrong" value of $h$ for finite-difference schemes. We know from Figure \ref{fig:bd} that for $T-100$, the desired confidence level of $0.95$ was achieved with $h=0.001$. This ``right" value of $h$ corresponds to Case 2 in Figure \ref{fig:tradeoff}. If we select a higher value of $h$, such as $h=0.1$ or $h=0.01$, then with $10000$ samples we see that the confidence level $p$ is much lower than the desired value $0.95$ (see Table \ref{tb:pitfalls}), indicating that the estimate is not sufficiently accurate. Note that $h=0.1$ and $h=0.01$ correspond to cases $1$ and $2$ in Figure \ref{fig:tradeoff}. Observe that for $h=0.1$, the estimate produced by finite-difference schemes ($\mu_N$) are way off the exact sensitivity value of $-9.995$, which shows that the bias can be large for a $h$ that appears small. On the other hand, if we select a $h$ which is really low, such as $h=0.0001$, then we do produce an estimate with the desired confidence level of $0.95$ (see Table \ref{tb:pitfalls}). However the required sample size is very large and hence the computational burden is much higher than the ``right" value of $h$, which is $h=0.001$. This value of $h$ corresponds to Case 4 in Figure \ref{fig:tradeoff}.

\begin{table}[h]
\caption{Birth-Death model - Pitfalls of choosing a ``wrong" $h$ } 
\label{tb:pitfalls} 
\begin{center} 
\begin{tabular}{|c  | c | c | c | c | c |c|}
\hline
$h$ &  Method  & $N$ & Mean ($\mu_N$) & Std. Dev. ($\sigma_N$) & CPU time (s) & p \\ \hline 
\multirow{2}{*}{$0.1$}
 &  CFD & 10000 &-4.8705 & 0.0585  & 0.14 & 0 \\
 &CRP  & 10000 &-4.8865 & 0.0729 & 0.64 & 0  \\ \hline 
 
 \multirow{2}{*}{$0.01$}
 &  CFD & 10000  &-9.3509 & 0.291174 & 0.16 & 0.3100 \\
& CRP & 10000 &-9.2809 & 0.311796  & 0.67 & 0.2459   \\ \hline 
 
 \multirow{2}{*}{$0.0001$}
 &   CFD  & 2578774   &-10.1676 & 0.1985& 36.3 & 0.95 \\
 &CRP & 1968305 &-9.8714 & 0.2238 & 130.44 & 0.95 \\ \hline 

\end{tabular}
\end{center}
\end{table}

}
\end{example}

\begin{example}[Gene Expression Network]
\label{ex:geneex} {\rm
Our second example considers the model for gene transcription given in \cite{MO}. It has three species : Gene ($G$), mRNA ($M$) and protein ($P$), and there are four reactions given by
\begin{align*}
G \stackrel{\theta_1 }{\rightarrow} G + M , \  \  M \stackrel{ \theta_2 }{\rightarrow} M + P,  \  \ M \stackrel{ \theta_3 }{\rightarrow} \emptyset  \   \textrm{ and } P \stackrel{ \theta_4}{\rightarrow} \emptyset.
\end{align*}
The rate of translation of a gene into mRNA is $\theta_1$ while the rate of transcription of mRNA into protein is $\theta_2$. The degradation of mRNA and protein molecules occurs at rates $\theta_3$ and $\theta_4$ respectively. Typically $\theta_3 \gg \theta_4$ implying that a protein molecule lives much longer than a mRNA molecule. In accordance with the values given in \cite{MO} for \emph{lacA} gene in bacterium \emph{E.Coli}, we set  $\theta_1 = 0.6 \ \mathrm{min}^{-1}$, $\theta_2 = 1.7329 \ \mathrm{min}^{-1}$ and $\theta_3 = 0.3466 \ \mathrm{min}^{-1}$. Our sensitive parameter is $\theta = \theta_4$.

Let $(X(t))_{t \geq 0}$ be the $\N^2_0$-valued Markov process representing the reaction dynamics. For any time $t$, $X(t) = (X_{1}(t) , X_{2}(t))$, where $X_{1}(t)$ and $X_{2}(t)$ are the number of mRNA and protein molecules respectively. We assume that $(X_{1}(0) , X_{2}(0)) = (0,0)$ and define $f : \N^2_0 \to \R$ by $f(x_1,x_2) = x_2$. We would like to estimate
\begin{align}
\label{sens_example2}
S_\theta(f,T) = \frac{\partial  }{\partial \theta} \E \left( f(X(T)) \right) = \frac{\partial  }{\partial \theta} \E ( X_{2}(T) ),
\end{align}
which measures the sensitivity of the mean of the protein population at time $T$ with respect to the protein degradation rate. 

For sensitivity estimation, we consider two values of $T$ : $20 \ \mathrm{min}$ and $100 \ \mathrm{min}$, and three values of $\theta$:   $0.0693 \ \mathrm{min}^{-1}$, $0.0023 \ \mathrm{min}^{-1}$ and $0$. These values of $\theta$ correspond to the protein half-life of $10 \ \mathrm{min}$, $5 \ \mathrm{hr}$ and $\infty$. Like Example \ref{ex:birth_death}, this network also has affine propensity functions and hence we can compute $S_\theta(f,T)$ exactly (see Section \ref{sec:data} of \emph{Supplementary Materials}). These exact values help in computing the confidence level of an estimate.

The CPU times required by all the sensitivity estimation methods for two threshold confidence levels ($p=0.95$ and $p=0.99$) are presented in Figure \ref{fig:ge}. Note that for $\theta = 0$, the Girsanov method cannot be used (see Section \ref{subsec:unbsdschemes}) but other methods can still be used. From Figure \ref{fig:ge} it is immediate that PPA can be far more efficient than the Girsanov method. As in the previous example, we compute the average speed-up factor for PPA with respect to other methods. In this example, these factors are calculated by aggregating CPU times for the three values of $\theta$ and the two values of $T$. These average speed-up factors are presented in Table \ref{ge:speedup}. Observe that in this example PPA is significantly faster than the Girsanov method. Furthermore, unlike PPA, the performance of the Girsanov method deteriorates drastically as $\theta$ gets smaller : on average PPA is $71$ times faster for $\theta = 0.0693 \ \mathrm{min}^{-1}$ and $2781$ times faster for $\theta = 0.0023  \ \mathrm{min}^{-1}$. This deterioration in the performance of the Girsanov method is consistent with the results in \cite{Our}. 

Note that in this example PPA is only slightly faster than the finite-difference schemes. However unlike finite-difference schemes, the unbiasedness of PPA guarantees that one does not require additional checks to verify its accuracy. 

\begin{figure}[h!]
  \begin{center}
    \includegraphics[height=10.0cm,width=17cm]{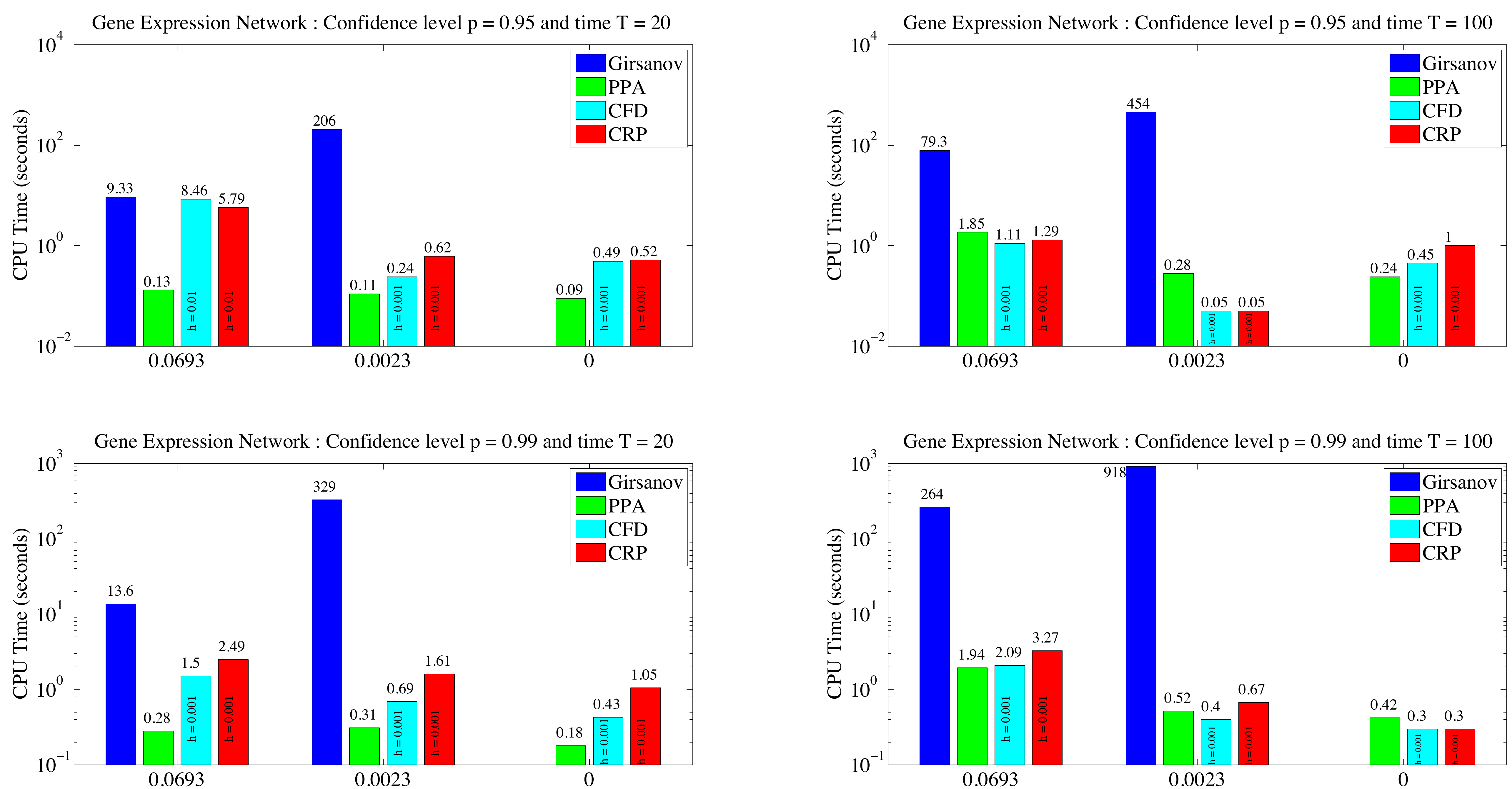}
  \end{center}
  \caption{Efficiency comparison for Gene Expression network with two threshold confidence levels : $p=0.95$ and $p=0.99$ and two observation times : $T=20$ min and $T=100$ min. Note that when the sensitive parameter $\theta$ is $0$, the Girsanov method is unusable and hence its CPU time is not displayed.}
\label{fig:ge}
\end{figure}

\begin{table}[h]
\caption{Gene Expression network - Average speed-up factors of PPA w.r.t other methods  } 
\label{ge:speedup}
\begin{center}
\begin{tabular}{| c |c | c | c |}
\hline
 $p$ & Girsanov & CFD & CRP  \\ \hline 
  0.95 & 315.9 & 4 & 3.4\\
  0.99 & 500 & 1.5 &  2.6\\ \hline
\end{tabular}
\end{center}
\end{table}
}
\end{example}

\begin{example}[Circadian Clock Network]
\label{ex:ccnet} {\rm
We now consider the network of a circadian clock introduced and studied by Vilar et. al. \cite{Vilar}. It is a large network with $9$ species $S_1,\dots,S_9$ and 16 reactions. These reactions along with the rate constants are given in Table \ref{tab:reaccc}.
\begin{table}[h]
\caption{Circadian Clock network  : List of reactions} 
\label{tab:reaccc} 
\centering
\begin{tabular}{|c|c|c|c|c|c|}
\hline
No. & Reaction & Rate Constant & No. & Reaction & Rate Constant \\ \hline
1 & $S_6 + S_2 \longrightarrow S_7 $  &  $\theta_1 = 1 $& 9 & $S_2 \longrightarrow \emptyset $ &  $\theta_9 = 100 $\\
2 & $S_7 \longrightarrow  S_6 + S_2 $ &  $\theta_2 = 50 $ &10 & $S_9 \longrightarrow S_9 + S_3  $&  $\theta_{10} = 0.01 $ \\
3 & $S_8 +S_2 \longrightarrow S_9 $&  $\theta_3 = 50 $ & 11 & $S_8\longrightarrow S_8 + S_3  $ &  $\theta_{11} = 50 $ \\
4 & $S_9 \longrightarrow   S_8 + S_2 $&  $\theta_4 = 500 $ & 12 & $ S_3 \longrightarrow \emptyset $ &  $\theta_{12} = 0.5 $ \\
5 & $S_7  \longrightarrow S_7+ S_1 $ &  $\theta_5 = 10 $ & 13 & $S_3 \longrightarrow S_3 + S_4  $  &  $\theta_{13} = 5 $\\
6 & $S_6  \longrightarrow S_6+ S_1  $ &  $\theta_6 = 50 $ & 14& $ S_4 \longrightarrow \emptyset $  &  $\theta_{14} = 0.2 $ \\
7 & $S_1 \longrightarrow \emptyset$ &  $\theta_7 = 1 $ & 15 & $S_2 + S_4 \longrightarrow S_5$&  $\theta_{15} = 20 $ \\
8 & $S_1 \longrightarrow S_1 + S_2 $ &  $\theta_8 = 1 $ & 16 & $S_5 \longrightarrow S_4 $ &  $\theta_{16} = 1 $ \\
\hline
\end{tabular}
\end{table} 
Let $(X(t))_{t \geq 0}$ be the $\N_0^d$-valued Markov process corresponding to the reaction dynamics with initial state $X(0) = (1,0,0,0,1,0,0,0,0)$.
 For $T = 5$ and $f(x) = x_4$ we wish to estimate
\begin{align*}
S_\theta(f,T) = \frac{\partial  }{\partial \theta} \E \left( f(X(T)) \right) = \frac{\partial  }{\partial \theta} \E ( X_4(T) ),
\end{align*} 
for $\theta = \theta_i$ and $i =5,6,8,12$ and $14$.

Due to the presence of many bimolecular interactions, exact values of the parameter sensitivity $S_\theta(f,T)$ are difficult to compute analytically, but they are required in computing the confidence levels of estimates. To solve this problem, we use PPA to produce sensitivity estimates with very low variance, and then use these approximate values in computing confidence levels. These approximate values are given in Section \ref{sec:data} of \emph{Supplementary Materials}.

The CPU times needed by all the sensitivity estimation methods for two threshold confidence levels ($p=0.95$ and $p=0.99$) are shown in Figure \ref{fig:cc}. 
As in the previous example, these results indicate that PPA has much better performance than the Girsanov method. 
In fact in many cases, the Girsanov method could not produce an estimate with the desired confidence level even with the maximum sample size of $10^7$. These values are marked with $*$ in Figure \ref{fig:cc} and the confidence level they correspond to is shown in parenthesis.  
\begin{figure}[h!]
  \begin{center}
    \includegraphics[height=10.0cm]{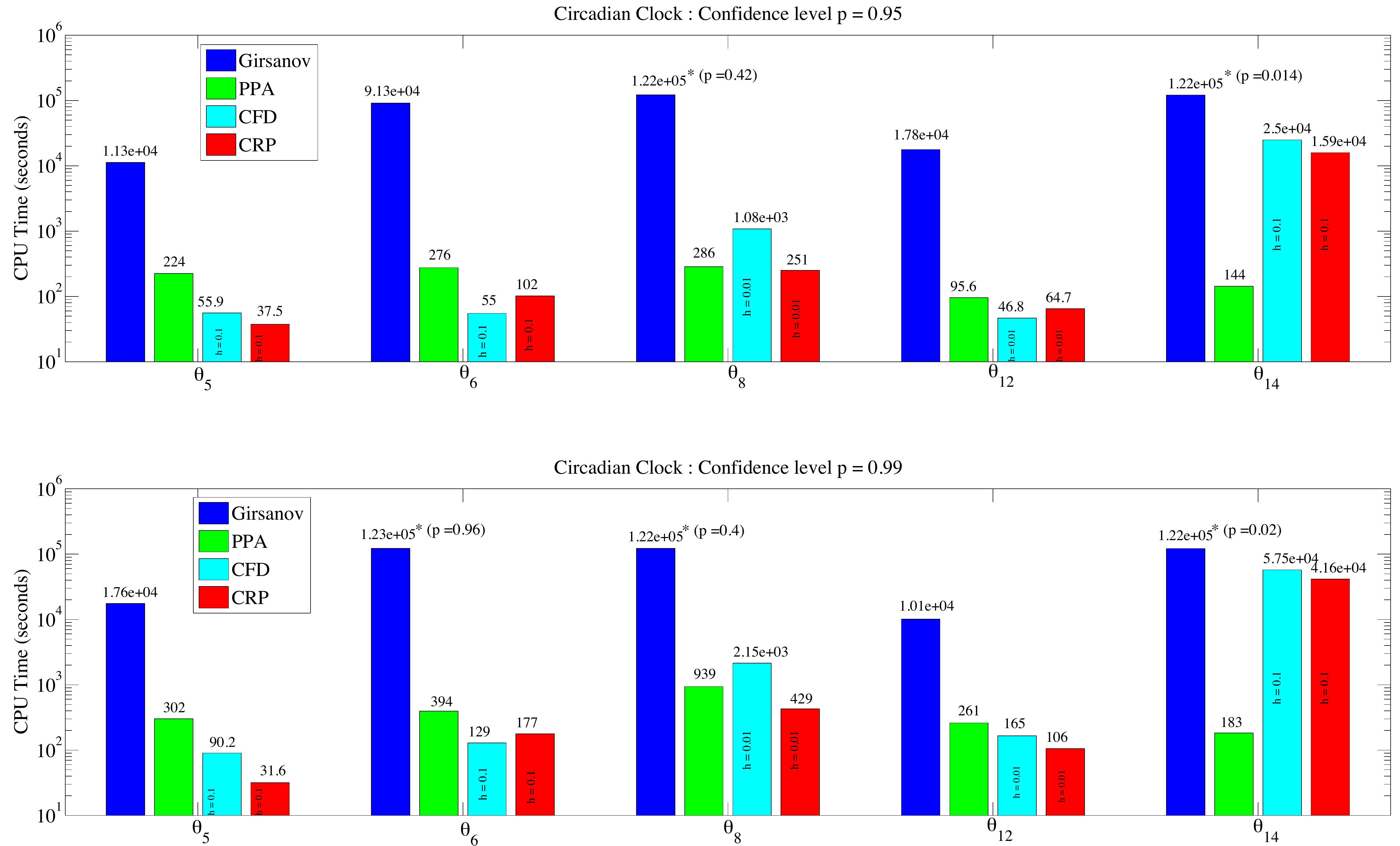}
  \end{center}
  \caption{Efficiency comparison for Circadian Clock network with two threshold confidence levels : $p=0.95$ and $p=0.99$ and observation time $T=5$. The values marked with $*$ correspond to cases where the Girsanov method failed to produce an estimate with the desired confidence level even with the maximum sample size of $10^7$. The final confidence level reached is shown in parenthesis.}
\label{fig:cc}
\end{figure}

If we compare the performance of PPA with finite-difference schemes, then we find that for some parameter values PPA is faster while for others it is slower. In the worst-case PPA is slower by a factor of $6$ but in the best-case PPA is faster by a factor of more than $300$. The overall efficiency of PPA can be compared from its average speed-up factor with respect to other methods. In this example, these factors are calculated by aggregating CPU times for the five values of $\theta$. These average speed-up factors are displayed in Table \ref{cc:speedup}. We mentioned before that in many cases the Girsanov method did not yield an estimate with the desired confidence level even with the maximum sample size. Hence the actual speed-up factor of PPA  with respect to the Girsanov method may be much higher than what is shown in Table \ref{cc:speedup}.
\begin{table}[h]
\caption{Circadian Clock network - Average speed-up factors of PPA w.r.t other methods  } 
\label{cc:speedup}
\begin{center}
\begin{tabular}{| c |c | c | c |}
\hline
 $p$ & Girsanov & CFD & CRP  \\ \hline 
  0.95 & 355 & 25 & 15.9\\
  0.99 & 190 & 28 &  20\\ \hline
\end{tabular}
\end{center}
\end{table}
}
\end{example}

\begin{example}[Genetic toggle switch]
\label{ex:tgsnet} { \rm
As our last example we look at a simple network with nonlinear propensity functions. Consider the network of a genetic toggle switch proposed by Gardner et. al. \cite{Gardner}. This network has two species $\mathcal{U}$ and $\mathcal{V}$ that interact through the following four reactions 
\begin{align*}
\emptyset \stackrel{\lambda_1 }{\rightarrow} \mathcal{U} , \  \  \mathcal{U} \stackrel{ \lambda_2 }{\rightarrow} \emptyset,  \  \ \emptyset \stackrel{\lambda_3 }{\rightarrow} \mathcal{V}  \   \textrm{ and } \mathcal{V}  \stackrel{ \lambda_4}{\rightarrow} \emptyset,
\end{align*}
where the propensity functions $\lambda_i$-s are given by
\begin{align*}
\lambda_1(x_1,x_2) = \frac{\alpha_1}{ 1 +x_2^{\beta} }, \ \ \lambda_2(x_1,x_2) = x_1 , \ \  \lambda_3(x_1,x_2) = \frac{\alpha_2}{ 1 +x_1^{\gamma} } \   \textrm{ and } \ \ \lambda_4(x_1,x_2) = x_2. 
\end{align*}
In the above expressions, $x_1$ and $x_2$ denote the number of molecules of $\mathcal{U}$ and $\mathcal{V}$ respectively. We set $\alpha_1 =50$, $\alpha_2 = 16$, $\beta = 2.5$ and $\gamma = 1$. 
Let $(X(t))_{t \geq 0}$ be the $\N^2_0$-valued Markov process representing the reaction dynamics with initial state $(X_{1}(0) , X_{2}(0)) = (0,0)$. For $T = 10$ and $f(x) = x_1$, our goal is to estimate
\begin{align*}
S_\theta(f,T) = \frac{\partial  }{\partial \theta} \E \left( f(X(T)) \right) = \frac{\partial  }{\partial \theta} \E ( X_1(T) ),
\end{align*} 
for $\theta = \alpha_1,\alpha_2,\beta$ and $\gamma$. In other words, we would like to measure the sensitivity of the mean of the number of $\mathcal{U}$ molecules at time $T=10$, with respect to all the model parameters.

Since the form of the propensity functions $\lambda_1$ and $\lambda_3$ is nonlinear, it is difficult to compute $S_\theta(f,T)$ exactly. As in the previous example, we obtain a close approximation of $S_{\theta}(f,T)$ using PPA. These approximate values are given in Section \ref{sec:data} of \emph{Supplementary Materials}, and they are used in computing confidence levels of estimates.

The CPU times needed by all the sensitivity estimation methods for two threshold confidence levels ($p=0.95$ and $p=0.99$) are presented in Figure \ref{fig:gts}. It can be easily seen that PPA is the best-performing method. 
\begin{figure}[h!]
  \begin{center}
    \includegraphics[height=10.0cm]{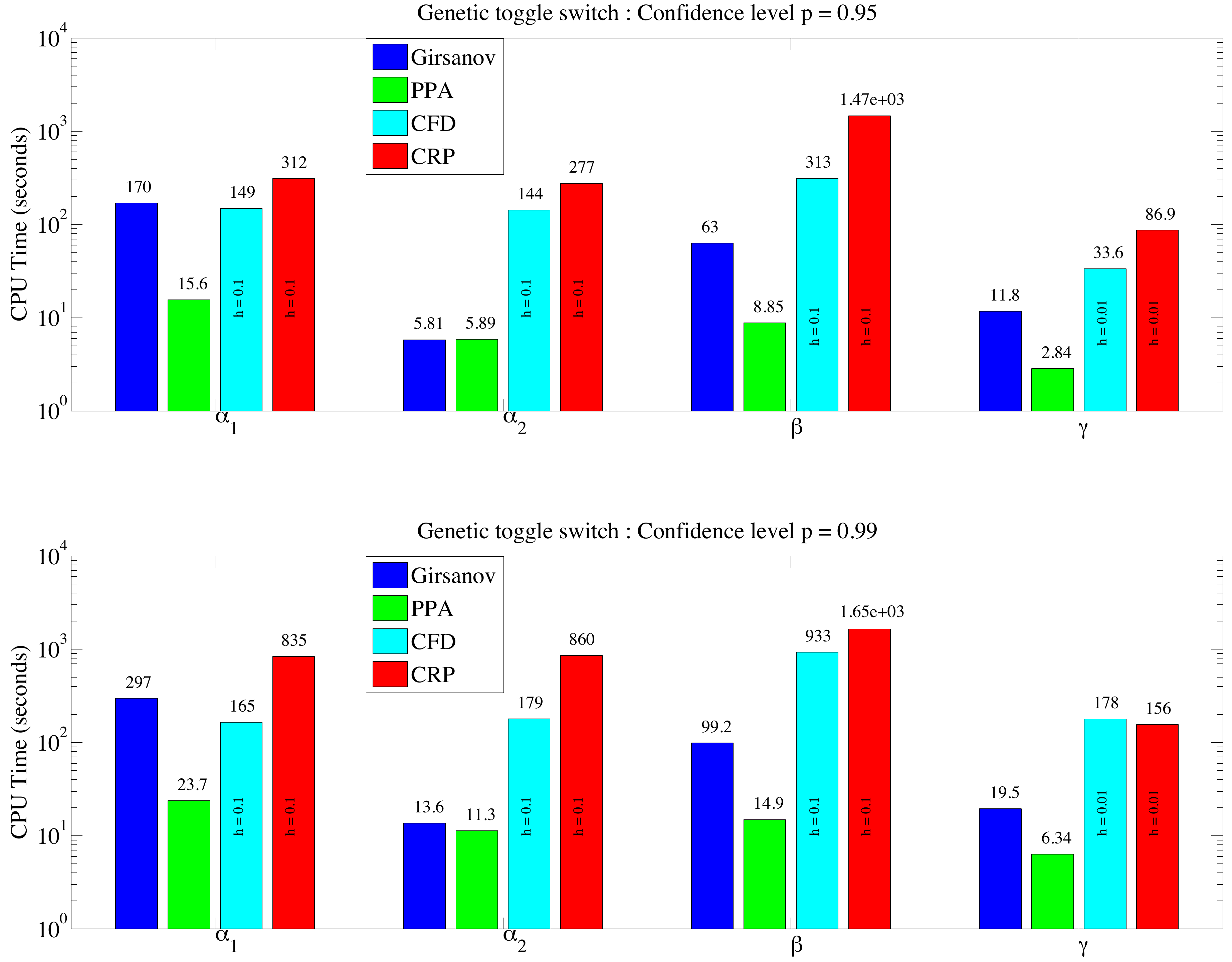}
  \end{center}
  \caption{Efficiency comparison for Genetic toggle switch with two threshold confidence levels : $p=0.95$ and $p=0.99$ and observation time $T=10$.}
\label{fig:gts}
\end{figure}

The average speed-up factors for PPA with respect to other methods are given in Table \ref{gts:speedup}. These factors are calculated by aggregating CPU times for the four values of $\theta$. Note that in this example, the unbiased schemes perform much better than the finite-difference schemes, possibly due to nonlinear parameter-dependence of the propensity functions.
\begin{table}[h]
\caption{Genetic toggle switch - Average speed-up factors of PPA w.r.t other methods  } 
\label{gts:speedup}
\begin{center}
\begin{tabular}{| c |c | c | c |}
\hline
 $p$ & Girsanov & CFD & CRP  \\ \hline 
  0.95 & 7.6 & 19.3 & 64 \\
  0.99 & 7.6 & 25.8 &  62.2\\ \hline
\end{tabular}
\end{center}
\end{table}
}
\end{example}

The examples presented in this section illustrate that in producing an estimate with a specified statistical accuracy, PPA can be much more efficient than other methods, both biased and unbiased. The finite-difference schemes may outperform PPA in some cases but the degree of outperformance is relatively small. Since it is difficult to independently verify the statistical accuracy of finite-difference schemes, PPA is an attractive alternative to these schemes, as its unbiasedness presents a theoretical guarantee for statistical accuracy.

 \section{Conclusions and Future Work} \label{sec:conc} 
 
 The aim of this paper is to provide a new unbiased method for estimating parameter sensitivity of stochastic reaction networks. Using a result from our recent paper \cite{Our}, we construct a random variable whose expectation is the required sensitivity value. We then present a simple procedure, called the Poisson Path Algorithm (PPA), to compute the realizations of this random variable. This gives us a way to generate samples for estimating the parameter sensitivity. Our method can be viewed as an improved version of the Auxiliary Path Algorithm
(APA), that we presented in \cite{Our}. Unlike APA, the proposed method is easy to implement, has low memory requirements and it works well for large networks and large observation times.

Through examples we compare the performance of PPA with other methods for sensitivity estimation, both biased and unbiased. Our results indicate that PPA easily outperforms the unbiased Girsanov method. Moreover in many cases it can be significantly faster that the best-performing finite-difference schemes (CRP and CFD) in producing a statistically accurate estimate. This makes PPA an appealing method for sensitivity estimation because it is computationally efficient and one does not have to tolerate a bias of an unknown size, that is introduced by finite-difference approximations.

In our method we simulate the paths of the underlying Markov process using Gillespie's Stochastic Simulation Algorithm (SSA) \cite{GP}. However SSA is very meticulous in the sense that it simulates the occurrence of each and every reaction in the dynamics. This can be very cumbersome for large networks with many reactions. To resolve this problem, a class of $\tau$-\emph{leaping} methods have been developed (see \cite{tleap1,tleap2}). These methods are approximate in nature but they significantly reduce the computational effort that is required to simulate the reaction dynamics. In future we would like to develop a version of PPA that uses a $\tau$-\emph{leaping} method instead of SSA and still produces an accurate estimate for the sensitivity values. Such a method would greatly simplify the sensitivity analysis of large networks.

\bibliographystyle{abbrv}



\newpage

\begin{center}
{ \LARGE \bf Supplementary Materials}
\end{center}

\renewcommand {\theequation}{A.\arabic{equation}}
\appendix
\setcounter{equation}{0}

\section{Proof of relation \eqref{expispreserved}} \label{asec:proof}

In this section we prove the main result on which our sensitivity estimation method is based.
\begin{proposition}
\label{prop:main}
Let $(X_\theta(t))_{t \geq 0}$ be the Markov process with initial state $x_0$, generator $\mathbb{A}_\theta$ and jump times $\sigma_0,\sigma_1,\dots$. Let $s_\theta(f,T)$ and $\hat{s}_\theta(f,T)$ be the random variables given by \eqref{expr:stheta} and \eqref{expr:sthetahat} respectively. Then \eqref{expispreserved} is satisfied for any $f: \N^d_0 \to \R$ and $T \geq 0$.
 \end{proposition}   
  \begin{proof}
  Let $\{\mathcal{F}_t\}_{t \geq 0}$ be the family of filtrations generated by $(X_\theta(t))_{t \geq 0}$. In other words, $\mathcal{F}_t$ is the sigma field contaning information about the process $X_\theta$ until time $t$. Note that the random variables $\eta$ and $\alpha$ are measurable with respect to $\mathcal{F}_T$. Moreover for every $i=1,\dots,\eta$, the random variables $\sigma_i, \Delta t_i, X_\theta(\sigma_i)$ and $\beta_{ki}$ are also $\mathcal{F}_T$-measurable.
  
Given $X_\theta(\sigma_i)$, the random variable $\gamma_i$ is independent of $\mathcal{F}_T$ and it is exponentially distributed with parameter 
$\lambda_0(X_\theta(\sigma_i) ,\theta)$. This implies that
\begin{align}
\label{propproof:exp0}
\E\left(  \frac{  \Gamma_i }{ \lambda_0( X_\theta (  \sigma_i  )  ,\theta ) }   \middle\vert \mathcal{F}_T \right) & = \frac{ 1}{ \lambda_0( X_\theta (  \sigma_i  )  ,\theta ) }  \P\left( \gamma_i < T - \sigma_i  \middle\vert \mathcal{F}_T  \right) = \int_{0}^{T - \sigma_i} e^{-\lambda_0(X_\theta(\sigma_i) ,\theta) s }ds,
\end{align}  
and from \eqref{diff_estimation} we obtain
\begin{align}
\label{propproof:exp1}
\E\left( \Gamma_i  \hat{D}_{ki} \middle\vert \mathcal{F}_T \right) & = \int_{0}^{T - \sigma_i} \lambda_0(X_\theta(\sigma_i) ,\theta) D_\theta( X_\theta(\sigma_i),f, T - \sigma_i- s,k ) e^{-\lambda_0(X_\theta(\sigma_i) ,\theta) s }ds.
 \end{align}
 Given $X_\theta(\sigma_i)$, $\rho^{(c)}_{ki} $ is a Poisson random variable independent of $\mathcal{F}_T$ and hence
 \begin{align}
 \label{propproof:exp2}
\E\left( \rho^{(c)}_{ki}  \middle\vert \mathcal{F}_T  \right) = \frac{c}{ \lambda_0( X_\theta(\sigma_i) ,\theta ) }  \left|  \frac{ \partial \lambda_k (X_\theta(\sigma_i) ,\theta ) }{ \partial \theta }  \right|.
\end{align}
Observe that $ \rho^{(c)}_{ki}  $ and $\Gamma_i  \hat{D}_{ki}$ are conditionally independent given $\mathcal{F}_T $.
Combining \eqref{propproof:exp1} and \eqref{propproof:exp2} we get
\begin{align}
\label{propproof:exp3}
\E\left( \frac{1}{c}  \beta_{ki}  \Gamma_i \rho^{(c)}_{ki} \hat{D}_{ki}  \middle\vert \mathcal{F}_T    \right)  & = \frac{\beta_{ki}  }{c} \E\left( \rho^{(c)}_{ki}  \middle\vert \mathcal{F}_T  \right)\E\left( \Gamma_i  \hat{D}_{ki} \middle\vert \mathcal{F}_T \right) \notag \\
 & =  \frac{\beta_{ki}  }{c} \left( \frac{c}{ \lambda_0( X_\theta(\sigma_i) ,\theta ) }  \left|  \frac{ \partial \lambda_k (X_\theta(\sigma_i) ,\theta ) }{ \partial \theta }  \right|  \right) \notag \\ & \ \  \left(  \int_{0}^{T - \sigma_i} \lambda_0(X_\theta(\sigma_i) ,\theta) D_\theta( X_\theta(\sigma_i),f, T - \sigma_i- s,k ) e^{-\lambda_0(X_\theta(\sigma_i) ,\theta) s }ds\right) \notag\\
 & =  \frac{ \partial \lambda_k (X_\theta(\sigma_i) ,\theta ) }{ \partial \theta } \int_{0}^{T - \sigma_i}  D_\theta( X_\theta(\sigma_i),f, T - \sigma_i- s,k ) e^{-\lambda_0(X_\theta(\sigma_i) ,\theta) s }ds \notag \\
  & =  \frac{ \partial \lambda_k (X_\theta(\sigma_i) ,\theta ) }{ \partial \theta } R_\theta(  X_\theta(\sigma_i),f, T - \sigma_i,k  ) \notag \\
  & + \frac{ \partial \lambda_k (X_\theta(\sigma_i) ,\theta ) }{ \partial \theta } \Delta_{\zeta_k} f( X_\theta(\sigma_i) ) \int_{0}^{T - \sigma_i} e^{-\lambda_0(X_\theta(\sigma_i) ,\theta) s }ds,
\end{align}
where $ R_\theta$ is defined by \eqref{defn_rtheta}.

On the event $\{ \alpha =1\}$, the state $X_\theta(\sigma_\eta)$ is absorbing ($\lambda_0( X_\theta(\sigma_\eta) ,\theta) = 0$). Hence $\Psi_\theta( X_\theta(\sigma_\eta) , f, s) = f( X_\theta(\sigma_\eta) )$ for all $s \geq 0$. From \eqref{int_estimation} we get
\begin{align}
\label{propproof:exp4}
 \E\left( \hat{I}_{ k \eta } \middle\vert \mathcal{F}_T   \right)  \notag & =   \int_{0}^{T - \sigma_\eta} \Psi_\theta( X_\theta(\sigma_\eta) +\zeta_k, f, s)ds \notag \\
& =  \int_{0}^{T - \sigma_\eta} D_\theta( X_\theta(\sigma_\eta),f,T - \sigma_\eta -s, k)ds +  (T - \sigma_\eta) f ( X_\theta (  \sigma_\eta  ) )  \notag \\
& = R_\theta(  X_\theta(\sigma_\eta),f, T - \sigma_\eta,k  ) 
+ \Delta_{\zeta_k} f( X_\theta(\sigma_\eta) )  \Delta t_{\eta} +  (T - \sigma_\eta) f ( X_\theta (  \sigma_\eta  ) ). 
\end{align}
 Taking conditional expectations in \eqref{expr:sthetahat} yields
\begin{align*}
&\E\left( \hat{s}_\theta(f,T) \middle\vert \mathcal{F}_T \right) = \sum_{k=1}^K \sum_{i=0}^{\eta - \alpha } \left[   \frac{ \partial \lambda_k ( X_\theta (  \sigma_i  ) , \theta ) }{ \partial \theta   } \Delta_{\zeta_k} f(X_\theta(  \sigma_i  )) \left( \Delta t_i - \E\left( \frac{  \Gamma_i }{ \lambda_0( X_\theta (  \sigma_i  )  ,\theta ) }  \middle\vert \mathcal{F}_T \right)  \right) \right. \\ & \left. + \E \left( \frac{1}{c}  \beta_{ki}  \Gamma_i \rho^{(c)}_{ki} \hat{D}_{ki}  \middle\vert \mathcal{F}_T \right) \right]  + \alpha \sum_{k=1}^K  \left[   \frac{ \partial \lambda_k ( X_\theta (  \sigma_\eta  ) , \theta ) }{ \partial \theta   } \left( \E \left( \hat{I}_{ k \eta } \middle\vert \mathcal{F}_T \right)   - (T - \sigma_\eta) f ( X_\theta (  \sigma_\eta  ) ) \right)\right]  .  \notag
\end{align*}
Using \eqref{propproof:exp0}, \eqref{propproof:exp3} and \eqref{propproof:exp4} we obtain
\begin{align*}
& \E\left( \hat{s}_\theta(f,T) \middle\vert \mathcal{F}_T \right) \\
& = \sum_{k=1}^K \sum_{i=0}^{\eta - \alpha } \left[   \frac{ \partial \lambda_k ( X_\theta (  \sigma_i  ) , \theta ) }{ \partial \theta   } \Delta_{\zeta_k} f(X_\theta(  \sigma_i  )) \left( \Delta t_i - \int_{0}^{T - \sigma_i} e^{-\lambda_0(X_\theta(\sigma_i) ,\theta) s }ds  \right) \right. \\ & \left. + 
\frac{ \partial \lambda_k (X_\theta(\sigma_i) ,\theta ) }{ \partial \theta } R_\theta(  X_\theta(\sigma_i),f, T - \sigma_i,k  ) + \frac{ \partial \lambda_k (X_\theta(\sigma_i) ,\theta ) }{ \partial \theta } \Delta_{\zeta_k} f( X_\theta(\sigma_i) ) \int_{0}^{T - \sigma_i} e^{-\lambda_0(X_\theta(\sigma_i) ,\theta) s }ds \right]  
\\& + \alpha \sum_{k=1}^K  \left[   \frac{ \partial \lambda_k ( X_\theta (  \sigma_\eta  ) , \theta ) }{ \partial \theta   } \left(  R_\theta(  X_\theta(\sigma_\eta),f, T - \sigma_i,k  ) 
+ \Delta_{\zeta_k} f( X_\theta(\sigma_i) )  \Delta t_{\eta} \right)\right] \\
& = \sum_{k=1}^K \sum_{i=0}^{\eta } \left[   \frac{ \partial \lambda_k ( X_\theta (  \sigma_i  ) , \theta ) }{ \partial \theta   } \Delta_{\zeta_k} f(X_\theta(  \sigma_i  ))  \Delta t_i + \frac{ \partial \lambda_k (X_\theta(\sigma_i) ,\theta ) }{ \partial \theta } R_\theta(  X_\theta(\sigma_i),f, T - \sigma_i,k  )  \right] \\
& =  \sum_{k = 1}^K  \left( \int_{0}^T \frac{ \partial \lambda_k ( X_\theta (t) , \theta ) }{ \partial \theta   } \Delta_{\zeta_k} f(X_\theta(t)) dt
+ \sum_{  i = 0 : \sigma_i < T   }^{\infty}  \frac{  \partial  \lambda_k ( X_\theta(\sigma_{i}) ,\theta ) }{ \partial \theta}  R_{\theta}( X_\theta( \sigma_i) ,f, T -\sigma_i ,k) \right).
\end{align*}
The last expression is equal to $s_\theta(f,T)$ (see \eqref{expr:stheta}) and hence
\begin{align*}
\E\left(  \hat{s}_\theta(f,T)  \middle\vert \mathcal{F}_T \right) = s_\theta(f,T).
\end{align*}
Taking expectations on both sides proves this proposition.
  \end{proof}

\section{Methods used in PPA} \label{app:methods}

In this section we provide descriptions for certain methods used in PPA.

\begin{algorithm}[h]                     
\caption{Computes the next time increment ($\Delta t$) and reaction index $(k)$ for Gillespie's SSA}
\label{SSA}  
\begin{algorithmic}[1]
\Function{SSA}{$x$} 
\State Set $r_1 = \mathrm{rand}()$ , $r_2 = \mathrm{rand}()$ and $k = 0$
\If {$\lambda_0(x,\theta) > 0$}
	\State Calculate $\Delta t = -\log(r_1)/\lambda_0(x,\theta)$
	\State Set $s = 0$
	\While {$s < r_2$}
		\State Update $k \gets k + 1$
		\State Update $s \gets s + \lambda_k(x,\theta)/ \lambda_0(x,\theta) $
	\EndWhile
\Else
	\State $\Delta t = \infty$
\EndIf 
\State \Return $(\Delta t, k)$
\EndFunction
\end{algorithmic}
\end{algorithm}

\begin{algorithm}[h]                       
\caption{Computes an estimate of  $\E\left(  R_{\textnormal{tot}} \right)$ using $N_0$ simulations of the process $X_\theta$}
\label{estimatenormalization}  
\begin{algorithmic}[1]
\Function{Estimate-R-total}{$x_0,T$}    
\State Set $R = 0$
\For { $i = 1$ to $N_0$}
	\State Set $x = x_0$ and $t = 0$
	\While {$t < T$}
		\State Calculate $(\Delta t, k_0)$ = SSA$(x)$
		\If {$\lambda_0(x,\theta) >0$}
		\For { $k = 1$ to $K$}
			\State Update $R \gets R + \frac{1}{ \lambda_0(x,\theta) }\left| \frac{ \partial \lambda_k(x,\theta) }{ \partial \theta} \right|$
		\EndFor
		\EndIf
		\State Update $t \gets t +\Delta t$ and $x\gets x+ \zeta_{k_0}$ 
	\EndWhile
\EndFor 
\State \Return $ R/N_0 $
\EndFunction
\end{algorithmic}
\end{algorithm}

\begin{algorithm}[h]
\caption{ Used to evaluate $\hat{I}_{k\eta}$ given by \eqref{evaluateiketa} }                           
\label{evaluateintegralindependent}  
\begin{algorithmic}[1]
\Function{EvaluateIntegral}{$x,T_f$}     
\State Set $t = 0$ and $I = 0$
\While {$ t <  T_f $}
	\State Calculate $(\Delta t, k)$ = SSA$(x)$
	\State Update $\Delta t  \gets \min\{\Delta t, T_f -t\}$
	\State Update $I \gets I+ f(x)  \Delta t$
	\State Update $t \gets t + \Delta t$ and $x \gets x + \zeta_k$
\EndWhile 
\State \Return $I$
\EndFunction
\end{algorithmic}
\end{algorithm}  
\begin{algorithm}                     
\caption{ Generates a Poisson random variable with parameter $r$}     
\label{genpoissrv}  
\begin{algorithmic}[1]
\Function{GeneratePoisson}{$r$}       
\State Set $p = \exp(-r)$, $s = p$, $n = 0$ and $ u = rand()$
\While {$ u >  s$}
	\State Update $n \gets n + 1$, $p \gets \left( \frac{p r}{n} \right)$ and $s \gets s + p$
\EndWhile 
\State \Return $n$
\EndFunction
\end{algorithmic}
\end{algorithm}

\begin{algorithm}[h]  
\caption{ Used to evaluate $\hat{D}_{ki}$ given by \eqref{evaluatedki} }           
 \label{gendiffsample}
\begin{algorithmic}[1]
\Function{EvaluateCoupledDifference}{$x_1,x_2,T_f$}    
\State Set $t = 0$
\For {$k=1$ to $K$}
		\State Set $T_{ki} = 0$ and $P_{ki} = - \log(rand())$ for $i=1,2,3$
\EndFor

\While { $x_1 \neq x_2$ \textbf{ AND }  $t < T_f$  } 
	\For {$k=1$ to $K$}
		\State Set $A_{k1} = \lambda_0(x_1,\theta) \wedge \lambda_0(x_2,\theta)$, $A_{k2} = \lambda_0(x_1,\theta) - A_{k1} $ and $A_{k3} = \lambda_0(x_2,\theta) - A_{k1} $
		\State Set $\delta_{ki} = \left( \frac{P_{ki} - T_{ki} }{  A_{ki} } \right)$ for $i=1,2,3$
	\EndFor
	\State Set $\Delta t = \min_{k,i} \{\delta_{ki} \}$ and $(k_{m} , i_{m}) = \argmin_{k,i} \{\delta_{ki}\}$
	\State Update $t \gets t + \Delta t $
	\If {$t <  T_f$}

	\If {$i_{m} = 1$ \textbf{ OR } $i_{m} = 2$ }
		\State Update $x_1 \gets x_1 + \zeta_{k_m }$ 
	\EndIf 
	\If {$i_{m} = 1$ \textbf{ OR } $i_{m} = 3$ }
		\State Update $x_2 \gets x_2 + \zeta_{k_m  }$ 
	\EndIf 

	\For {$k=1$ to $K$}
		\State Update $T_{ki} \gets T_{ki}  +  A_{ki} \Delta t$ for $i=1,2,3$ 
	\EndFor
	\State Update $P_{k_m i_m} \gets P_{k_m i_m}  - \log(rand())$
		\EndIf
\EndWhile
\State \Return $f(x_1)-f(x_2)$
\EndFunction
\end{algorithmic}
\end{algorithm}

\FloatBarrier

\section{Data for numerical examples} \label{sec:data}

In this section we provide all the data that was generated to compare the performance of various sensitivity estimation methods in Section \ref{sec:ex}. In the following tables, an estimate for parameter sensitivity produced by a method is described by the following quantities: the sample size $N$, the estimator mean $\mu_N$, the estimator variance $\sigma_N$ and the CPU time (in seconds). Here $N$ is the minimum number of samples needed to attain the threshold confidence level of $p=0.95$ or $p=0.99$. For finite-difference schemes (CFD and CRP), the corresponding value of $h$ is also shown. This value of $h$ is the largest in the sequence $10^{-1},10^{-2},10^{-3},\dots$ for which the desired confidence level was achieved.

\subsection{ Birth-death Network }

\begin{center}
{\bf Exact values of $S_\theta(f,T)$}

\begin{tabular}{|c | c | c |}
\hline
T & 20 & 100  \\ \hline 
$S_\theta(f,T)$ & -5.9399 & -9.995  \\ \hline
\end{tabular}
\end{center}

\begin{center}
{\bf Efficiency comparison for confidence level $p=0.95$}

\begin{tabular}{| c |c | c | c | c |c| c | }
\hline
 $T$ & Method & N  & Mean ($\mu_N$) & Std. Dev. ($\sigma_N$) & CPU time (s) & h \\ \hline 
   \multirow{4}{*}{20} & Girsanov & 24474 &-5.8508 & 0.1247 &  0.21 & --- \\
&PPA & 3119&-5.9753 & 0.1472  & 0.04 & --- \\ 
& CFD & 266016  &-5.9658 & 0.1493 & 2.13 & 0.001 \\
&CRP & 283127 &-5.8949 & 0.1444  & 16.83 & 0.001    \\ \hline

\multirow{4}{*}{100} &Girsanov & 67723 &-9.8111 & 0.1918 & 0.74 & ---  \\
 &PPA & 13567&-10.1301 & 0.2188  & 0.4 & --- \\ 
 &CFD& 325365 &-9.7798 & 0.1725  & 4.6 & 0.001 \\
 &CRP & 335946 &-9.7784 & 0.1706  & 21.15  & 0.001  \\ \hline

\end{tabular}
\end{center}

\begin{center}
{\bf Efficiency comparison for confidence level $p=0.99$}

\begin{tabular}{| c |c | c | c | c |c| c | }
\hline
 $T$ & Method & N  & Mean ($\mu_N$) & Std. Dev. ($\sigma_N$) & CPU time (s) & h \\ \hline 
   \multirow{4}{*}{20} &  Girsanov & 30645 &-5.9449 & 0.1152  & 0.29 & --- \\
&PPA & 5495 &-5.9712 & 0.1110  & 0.06 &  --- \\ 
 & CFD & 577229 &-5.8781 & 0.1006  & 5 & 0.001 \\
 &CRP& 474599 &-5.9103 & 0.1115  & 34.49 & 0.001 \\ \hline

\multirow{4}{*}{100} &Girsanov & 67627 &-9.9571 & 0.1903  & 0.78 & ---  \\
  &PPA& 19207 &-10.0568 & 0.1843  & 0.66 & --- \\ 
 &  CFD & 291565 &-10.0561 & 0.1848  & 4.08 & 0.001\\
  &CRP & 306807 &-10.0682 & 0.1810  & 22.7  & 0.001   \\ \hline

\end{tabular}
\end{center} 

\subsection{Gene Expression Network}

\begin{center} 
{\bf  Exact values of $S_\theta(f,T)$} 

\begin{tabular}{|c | c | c |}
\hline
 $\theta$ & T & $S_\theta(f,t)$  \\ \hline
 \multirow{2}{*}{$0.0693 $} & 20 & -207.544  \\ 
&  100 &  -618.776 \\ \hline

\multirow{2}{*}{$0.0023 $} & 20 & -439.601  \\ 
&  100 &  -12213.9 \\ \hline

\multirow{2}{*}{$0 $} & 20 & -451.812  \\ 
&  100 & -14158.6 \\ \hline  
\end{tabular}
\end{center}

\begin{center} 
{\bf Efficiency comparison for confidence level $p=0.95$} 
\begin{tabular}{|c | c | c | c | c | c | c |c |}
\hline
$\theta$ & T &  Method  & N  & Mean ($\mu_N$) & Std. Dev. ($\sigma_N$) & CPU time (s) & h \\ \hline 
\multirow{8}{*}{$0.0693 $}
& \multirow{4}{*}{20 }   &Girsanov & 346580 &-210.872 & 4.2674  & 9.33 & --- \\
 & &PPA & 563 &-207.406 & 5.2869  & 0.13 & ---  \\  
 & & CFD & 148730 &-197.711 & 0.3308  & 8.46  & 0.01  \\
 & & CRP & 39054  &-198.546 & 0.8375 & 5.79   & 0.01 \\  \cline{2-8}

& \multirow{4}{*}{100} &  Girsanov & 511629 &-623.925 & 14.9357  & 79.28 & --- \\
&  & PPA& 2259  &-606.18 & 10.9611 & 1.85 & --- \\
&  & CFD & 3297  &-602.062 & 8.6224 & 1.11  & 0.001\\
&  & CRP& 3227  &-611.094 & 13.7036  & 1.29  & 0.001 \\   \hline \hline

\multirow{8}{*}{$0.0023 $}
& \multirow{4}{*}{20 }   & Girsanov & 8923245 &-432.886 & 9.2280  & 206 &  --- \\
 & & PPA & 592 &-438.32 & 11.0762  & 0.11 &  --- \\
 & & CFD & 4972  &-445.093 & 9.85& 0.24 & 0.001 \\
 & & CRP & 3997&-440.08 & 11.2038  & 0.62  & 0.001  \\  \cline{2-8} 
 
& \multirow{4}{*}{100} &  Girsanov & 4473418 &-12186.3 & 310.358  & 454.49 &  ---\\
&  & PPA & 240 &-12150.2 & 304.914  & 0.28 &  --- \\   
&  & CFD& 241   &-12074.7 & 278.079 & 0.05 & 0.001 \\
& & CRP & 205 &-12063.4 & 272.183  & 0.05 & 0.001 \\   \hline \hline

\multirow{6}{*}{$0$}
& \multirow{3}{*}{20  }   & PPA  & 474 &-451.881 & 11.5112 & 0.09 &  --- \\
 & &CFD& 10381 &-440.42 & 6.7753  & 0.49 & 0.001 \\
 & &CRP& 3776 &-451.006 & 11.481  & 0.52  & 0.001\\  \cline{2-8} 
 
& \multirow{3}{*}{100}   & PPA & 252 &-14018.1 & 322.54  & 0.24 &  ---\\  
 & & CFD& 1912  &-13626.6 & 106.829  & 0.45 & 0.001 \\
&  &CRP & 3431&-13582 & 79.3105  & 1 & 0.001 \\   \hline \hline

\end{tabular}
\end{center} 

\begin{center}
{\bf Efficiency comparison for confidence level $p=0.99$} 
 
\begin{tabular}{|c | c | c | c | c | c | c |c |}
\hline
$\theta$ & T &  Method  & N  & Mean ($\mu_N$) & Std. Dev. ($\sigma_N$) & CPU time (s) & h \\ \hline 
\multirow{8}{*}{$0.0693 $}
& \multirow{4}{*}{20 }   & Girsanov& 472152   &-209.302 & 3.67417 & 13.64 &---  \\
 & &PPA  & 971&-207.548 & 4.0278 & 0.28 & --- \\
 &  &CFD  & 26041 &-202.988 & 2.4926 & 1.5 & 0.001 \\
 & &CRP & 15712 &-209.203 & 3.7109  & 2.49  & 0.001\\  \cline{2-8}

& \multirow{4}{*}{100} & Girsanov & 1665299  &-607.096 & 8.2767 & 263.76 & ---  \\
&  &PPA & 2284 &-613.211 & 10.7988  & 1.94 & --- \\   
&  & CFD& 6006  &-602.897 & 6.4532 & 2.09 & 0.001\\
&  & CRP & 7888 &-628.676 & 9.0198 & 3.27 & 0.001 \\   \hline \hline

\multirow{8}{*}{$0.0023 $}
& \multirow{4}{*}{20 }   & Girsanov & 13874198  &-444.325 & 7.3984 & 329 & ---   \\
 & &PPA& 1353  &-445.363 & 6.9574  & 0.31 & --- \\ 
 &  & CFD & 14448  &-448.297 & 5.7024 & 0.69 & 0.001\\
 & &CRP  & 10951 &-432.928 & 6.5652 & 1.61  & 0.001 \\  \cline{2-8} 
 
& \multirow{4}{*}{100} &   Girsanov& 8777267 &-12302.1 & 221.582  & 917.69 & ---\\
&  &PPA & 490  &-12146.3 & 224.752 & 0.52 & --- \\   
&  & CFD & 1450 &-11845.5 & 103.974  & 0.4  & 0.001 \\
&  & CRP & 1359 &-11863.1 & 111.482 & 0.67& 0.001 \\   \hline \hline

\multirow{6}{*}{$0$}
& \multirow{3}{*}{20  }   & PPA & 871 &-451.562 & 8.7611  & 0.18 & --- \\
 & &CFD & 8413  &-456.674 & 7.5868 & 0.43  & 0.001 \\
 & &CRP & 6770  &-449.483 & 8.3893 & 1.05  & 0.001\\  \cline{2-8} 
 
& \multirow{3}{*}{100}   & PPA& 424 &-14195.7 & 270.698  & 0.42 & --- \\  
 & & CFD& 924  &-13810.6 & 153.342 & 0.3  & 0.001 \\
&  &CRP & 953&-13804.8 & 150.552  & 0.3 & 0.001 \\   \hline \hline

\end{tabular}
\end{center}

\newpage

\subsection{Circadian clock network }

\begin{center} 
{\bf Approximate values of $S_\theta(f,T)$} 

\begin{tabular}{|c |c| c | c | c | c |}
\hline
 $\theta  $ &  $\theta_5 =10$  & $\theta_6 = 50$ & $\theta_8 = 1$ & $\theta_{12} = 0.5$ & $\theta_{14} = 0.2$  \\  \hline
$S_{\theta}(f,T)$ &  $-240.368$  & $47.0746  $ & $ -127.629 $ & $ 1469.81 $ & $0.1424 $  \\ \hline
\end{tabular}
\end{center}

\begin{center}
{\bf Efficiency comparison for confidence level $p=0.95$} 

\begin{tabular}{|c  | c | c | c | c | c |c |}
\hline
$\theta$ &  Method & N & Mean ($\mu_N$) & Std. Dev. ($\sigma_N$)  & CPU time (s) & h \\ \hline 

\multirow{4}{*}{$\theta_5 =10$} & Girsanov& 932283 &-237.142 & 5.2824  & 11292.9 & ---  \\
 & PPA & 1266 &-238.7 & 5.8722 & 224.38 & --- \\ 
 & CFD& 1028 &-234.193 & 3.5001  & 55.85  & 0.1  \\
 &CRP & 837  &-243.943 & 5.019 & 37.51 & 0.1 \\ \hline 
 
 \multirow{4}{*}{$\theta_6 =50$} & Girsanov & 7535847  &46.1122 & 0.8456 & 91253.5 & --- \\
 & PPA & 2312  &47.1316 & 1.199 & 276.15 & ---\\ 
 & CFD & 1034 &47.3501 & 1.1664  & 55.04 & 0.1 \\
  & CRP  & 3019  &46.9361 & 1.1921 & 102.26  & 0.1 \\ \hline 
   \multirow{4}{*}{$\theta_8 = 100$} & Girsanov & $10^7$ ($p=0.4197$)  &-120.295 & 7.1232  & 122393 & --- \\
 & PPA & 8627 &-127.668 & 3.2552  & 286.31 & --- \\ 
 & CFD & 20610  &-127.04 & 3.1996 & 1084.35 & 0.01 \\
 & CRP & 7199  &-128.296 & 3.1853 & 251.14  & 0.01 \\ \hline
 
\multirow{4}{*}{$\theta_{12} =0.5 $} & Girsanov & 1477492 &1429.5 & 20.1706  & 17750.3 & --- \\
 & PPA & 402  &1451.45 & 32.7907 & 95.61 & ---\\
& CFD & 856  &1484.81 & 34.4219 & 46.81 & 0.01   \\
 & CRP& 1902 &1424.87 & 16.986  & 64.71  & 0.01 \\ \hline 
 
 \multirow{4}{*}{$\theta_{14} = 0.2$} & Girsanov & $10^7$  ($p=0.0139$) &-0.1025 & 0.2748  & 121611 & ---\\
 &PPA & 707 &0.1421 & 0.0036 & 143.53 &--- \\
 & CFD  & 473002&0.1401 & 0.0029 & 24993.6 & 0.1 \\
 &CRP & 451383  &0.1419 & 0.0036 & 15893.1 & 0.1   \\ \hline
 
\end{tabular}
\end{center}

\begin{center} 
{\bf  Efficiency comparison for confidence level $p=0.99$} 

\begin{tabular}{|c  | c | c | c | c | c |c |}
\hline
$\theta$ &  Method & N & Mean ($\mu_N$) & Std. Dev. ($\sigma_N$)  & CPU time (s) & h \\ \hline  

\multirow{4}{*}{$\theta_5 =10$} &  Girsanov & 1450448  &-238.278 & 4.23555 & 17554.5 & --- \\
 &  PPA & 1613 &-240.477 & 4.663 & 302.38 & --- \\
 & CFD & 1610 &-234.522 & 2.651  & 90.21  & 0.1  \\
 & CRP & 908  &-239.075 & 4.446 & 31.59  & 0.1 \\ \hline 
 
 \multirow{4}{*}{$\theta_6 =50$}  & Girsanov & $10^7$ ($p=0.9589$)  &45.9936 & 0.732054 & 122605  & --- \\ 
 & PPA & 3181 &47.3226 & 0.8806  & 393.63 & --- \\ 
 &  CFD & 2079 &47.4844 & 0.8268   & 128.82  & 0.1 \\
  & CRP & 5014 &47.0243 & 0.9124  & 176.95  & 0.1 \\ \hline 
  
   \multirow{4}{*}{$\theta_8 = 100$} & Girsanov & $10^7$ ($p=0.3972$)  &-119.825 & 7.2083  & 122295 & --- \\ 
 & PPA & 28037 &-129.732 & 1.8377  & 939.13 & --- \\ 
 &  CFD& 39614 &-128.578 & 2.3064  & 2145.17  & 0.01  \\
 & CRP& 12300 &-128.122 & 2.4254  & 428.54 & 0.01 \\ \hline 
 
\multirow{4}{*}{$\theta_{12} =0.5 $}  &Girsanov& 825031  &1460.32 & 27.015  & 10126.4 & --- \\
 &PPA & 1075 &1445.52 & 21.1119  & 261.03 & --- \\
& CFD & 2873  &1439.44 & 18.4201 & 165.05  & 0.01  \\
 & CRP & 3061  &1428.16 & 13.6176 & 105.95 & 0.01\\ \hline 
 
 \multirow{4}{*}{$\theta_{14} = 0.2$}  &Girsanov  &  $10^7$ ($p=0.0205$) &0.171034 & 0.275163& 122064 & --- \\
 &PPA & 877  &0.1424 & 0.0028 & 182.81 & --- \\
 & CFD& 1062478  &0.1396 & 0.0018  & 57455.9 & 0.1 \\
 &CRP  & 1178757 &0.1444 & 0.0022 & 41566.7  & 0.1   \\ \hline 
 
\end{tabular}
\end{center}  

\newpage 
\subsection{Genetic toggle switch}
\begin{center}
{\bf Approximate values of $S_\theta(f,T)$} 

\begin{tabular}{|c |c| c | c | c | }
\hline
 $\theta  $ &  $\alpha_1 = 50$  & $\alpha_2 = 16$ & $\beta= 2.5$ & $\gamma = 1$  \\  \hline
$S_{\theta}(f,T)$ &  $1.19 $   & $-2.107   $ & $ -5.9571$ & $ 54.7495  $   \\ \hline
\end{tabular}
\end{center}

\begin{center} 
{\bf Efficiency comparison for confidence level $p=0.95$} 

\begin{tabular}{|c  | c | c | c | c | c |c |}
\hline
$\theta$ &  Method  & N & Mean ($\mu_N$) & Std. Dev. ($\sigma_N$)  & CPU time (s) & h \\ \hline 

\multirow{4}{*}{$\alpha_1 =50$}
& Girsanov & 490831  &1.1656 & 0.0213 & 170.23 & --- \\
 &  PPA  & 15125 &1.175 & 0.0263 & 15.57 & --- \\ 
 &  CFD & 358066 &1.1679 & 0.0225 & 149.37 &0.1\\
 &CRP & 693353 &1.1847 & 0.0299   & 311.71 &  0.1\\ \hline 
 
 \multirow{4}{*}{$\alpha_2 =16$}
 &Girsanov& 16911  &-2.0689 & 0.0406 & 5.81 & ---  \\
 &PPA& 5578 &-2.1063 & 0.0537  & 5.89 & --- \\   
 & CFD & 370794  &-2.1399 & 0.0438 & 143.56 &0.1  \\
 & CRP & 596643  &-2.0981 & 0.053 & 276.66  &0.1  \\ \hline 
  
   \multirow{4}{*}{$\beta =2.5$}
 &Girsanov & 174222  &-6.0027 & 0.145 & 62.99 & ---  \\
& PPA & 8898 &-5.9287 & 0.149  & 8.85 & ---  \\ 
 & CFD & 677286  &-5.7496 & 0.0547 & 313.34 &0.1  \\
 &CRP & 3109153  &-5.7104 & 0.0311 & 1466.56  & 0.1  \\ \hline 
 
\multirow{4}{*}{$\gamma = 1 $} 
&  Girsanov& 32762  &55.1442 & 1.3392  & 11.78 & --- \\
 & PPA & 1553 &54.8859 & 1.3883  & 2.84 & --- \\  
& CFD & 70730 &54.8282 & 1.3945  & 33.57  & 0.01  \\
 & CRP & 167092 &54.6585 & 1.3934  & 86.87 & 0.01  \\ \hline

\end{tabular}
\end{center}

\begin{center} 
{\bf Efficiency comparison for confidence level $p=0.99$} 

\begin{tabular}{|c  | c | c | c | c | c |c |}
\hline
$\theta$ &  Method  & N & Mean ($\mu_N$) & Std. Dev. ($\sigma_N$)  & CPU time (s) & h \\ \hline 

\multirow{4}{*}{$\alpha_1 =50$}
& Girsanov& 814460 &1.1691 & 0.0166  & 296.9 & --- \\
 &  PPA & 22757  &1.1938 & 0.0228& 23.72 & --- \\  
 & CFD & 401198 &1.1974 & 0.022  & 164.54 & 0.1\\
 & CRP & 1742305 &1.2053 & 0.019  & 834.77  & 0.1\\ \hline 
 
 \multirow{4}{*}{$\alpha_2 =16$}
 &Girsanov & 37611 &-2.1489 & 0.0272  & 13.61 & ---  \\
 &PPA & 10926  &-2.1215 & 0.0385& 11.33 & --- \\  
 & CFD & 431036 &-2.1143 & 0.04 & 178.77  & 0.1   \\
 & CRP & 1839908 &-2.1424 & 0.03 & 860.29 & 0.1  \\ \hline 
  
   \multirow{4}{*}{$\beta =2.5$}
 &Girsanov& 275232   &-5.9431 & 0.1147 & 99.17 & --- \\
&PPA & 18342  &-5.9233 & 0.1108 & 14.93 & --- \\
 & CFD & 2235576 &-5.7292 & 0.03 & 932.54 & 0.1  \\
 &CRP& 3401875   &-5.7285 & 0.0298 & 1650.65 & 0.1\\ \hline  
 
\multirow{4}{*}{$\gamma = 1 $} 
&  Girsanov & 54087  &54.5545 & 1.0448 & 19.51 & ---    \\
 & PPA& 3247  &54.2298 & 0.9452 & 6.34 & --- \\  
& CFD & 356293  &53.4473 & 0.6144 & 178.28 & 0.01  \\
 & CRP & 356293 &54.5786 & 1.0488 & 155.94 & 0.01   \\ \hline

\end{tabular}
\end{center}

\end{document}